\documentclass[11pt, reqno]{amsart}
\usepackage{amsmath, amsthm, amssymb, amsfonts, mathrsfs}
\usepackage[a4paper, left=3.2cm, right=3.2cm, top=3.2cm, bottom=3.2cm]{geometry}
\usepackage[hidelinks]{hyperref}
\numberwithin{equation}{section}

\title[Arithmetic dynamics of random polynomials]{Arithmetic dynamics of random polynomials}

\author{Pierre Le Boudec}
\address{Departement Mathematik und Informatik \\ Fachbereich Mathematik \\ Spiegelgasse 1 \\ 4051 Basel \\ Switzerland}
\email{pierre.leboudec@unibas.ch}

\author{Niki Myrto Mavraki}
\address{Harvard University \\ Department of Mathematics \\ Science Center Room 325 \\ 1 Oxford Street \\ Cambridge \\ MA 02138 \\ USA}
\email{mavraki@math.harvard.edu}

\subjclass[2010]{37P05, 37P15, 37P35}
\keywords{Polynomials, preperiodic points, canonical heights}

\begin{document}

\newtheorem{theorem}{Theorem}[section]
\newtheorem{lemma}[theorem]{Lemma}
\newtheorem{proposition}[theorem]{Proposition}
\newtheorem{corollary}[theorem]{Corollary}
\newtheorem{conjecture}[theorem]{Conjecture}
\newtheorem{conj}{Conjecture}
\renewcommand{\theconj}{\Alph{conj}}
\theoremstyle{definition}
\newtheorem{definition}[theorem]{Definition}

\newcommand{\rad}{\operatorname{rad}}

\begin{abstract}
We investigate from a statistical perspective the arithmetic properties of the dynamics of polynomials of fixed degree and defined over the field of rational numbers. To start with, ordering their affine conjugacy classes by height, we show that their average number of rational preperiodic points is equal to zero, thereby proving a strong average version of the uniform boundedness conjecture of Morton and Silverman. Next, inspired by the analogy with the successive minima of a lattice we define the dynamical successive minima of a polynomial. Noting that these quantities are invariant under the action by conjugacy of the affine group we study their average behaviour using the aforementioned ordering by height. In particular, we prove an optimal statistical version of the dynamical Lang conjecture on the canonical height of rational non-preperiodic points.
\end{abstract}

\maketitle

\thispagestyle{empty}
\setcounter{tocdepth}{1}
\tableofcontents

\section{Introduction}

The arithmetic properties of the dynamics of polynomials defined over the field of rational numbers, despite being intensely investigated, remain largely elusive. Indeed, even in the simplest case of quadratic polynomials, the uniform boundedness conjecture of Morton and Silverman \cite{MR1264933} on the number of rational preperiodic points, as well as the dynamical Lang conjecture \cite[Conjecture~$4$.$98$]{MR2316407} on the canonical height of rational non-preperiodic points are wide open. We note that the number of rational preperiodic points of a polynomial defined over $\mathbb{Q}$ and the set of values of its canonical height at rational points are both invariant under the action by conjugacy of the affine group $\mathbb{Q} \rtimes \mathbb{Q}^{\times}$. Given $d \geq 2$, it is thus natural to study these conjectures on average over affine conjugacy classes of polynomials of degree $d$.

Outside a set of dimension $d-2$, we can parametrize these conjugacy classes by a Zariski open subset of $\mathbb{P}^{d-1}(\mathbb{Q})$ as follows. We view projective points as primitive vectors of integers so we let $\mathbb{Z}_{\mathrm{prim}}^d$ be the set of $(x_1, \dots, x_d) \in \mathbb{Z}^d$ such that $\gcd(x_1, \dots, x_d)=1$, and we introduce the set
\begin{equation*}
\mathscr{P}_d = \left\{ \psi_{\mathbf{a}} \in \mathbb{Q}[z] :
\begin{array}{l l}
\mathbf{a} = (a_d, a_{d-2}, \dots , a_0) \in \mathbb{Z}^d_{\mathrm{prim}} \\
a_d \neq 0, \ a_0 > 0
\end{array}
 \right\},
\end{equation*}
where
\begin{equation}
\label{Definition psi}
\psi_{\mathbf{a}}(z) = \frac{a_d}{a_0} z^d + \frac{a_{d-2}}{a_0} z^{d-2} + \cdots + \frac{a_1}{a_0} z + 1.
\end{equation}
We remark that the affine conjugacy class of any degree $d$ polynomial which does not fix the barycenter of its roots in some algebraic closure of $\mathbb{Q}$ contains exactly one element of $\mathscr{P}_d$.

We shall order elements of $\mathscr{P}_d$ using the usual exponential height on projective space. We thus define the height $\mathscr{H}(\psi_{\mathbf{a}})$ of a polynomial $\psi_{\mathbf{a}} \in \mathscr{P}_d$ as
\begin{equation*}
\mathscr{H}(\psi_{\mathbf{a}}) = \max \left\{ |a_d|, |a_{d-2}|, \dots, |a_1|, a_0 \right\}.
\end{equation*}
Finally, given $X \geq 1$ we let
\begin{equation*}
\mathscr{P}_d(X) = \left\{ \psi_{\mathbf{a}} \in \mathscr{P}_d : \mathscr{H}(\psi_{\mathbf{a}}) \leq X \right\}.
\end{equation*}

We now proceed to introduce some classical notation. Given a set $S$ and a map $f : S \to S$, for any $n \geq 1$ we let $f^n$ denote the $n$-th iterate of $f$ and by convention we let $f^0$ be the identity map. Moreover, we let $\mathrm{Prep}_S(f)$ denote the set of preperiodic points of $f$ in $S$, that is
\begin{equation*}
\mathrm{Prep}_S(f) = \left\{ s \in S : \exists \ell \geq 0 \ \exists m \geq 1 \ f^{\ell+m}(s) = f^{\ell}(s) \right\}.
\end{equation*}

The following conjecture is a particular case of the uniform boundedness conjecture of Morton and Silverman \cite{MR1264933}.

\begin{conj}[Morton--Silverman]
\label{Conjecture UBC}
Let $d \geq 2$. There exists $C_d > 0$ such that for any $\psi_{\mathbf{a}} \in \mathscr{P}_d$, we have
\begin{equation*}
\# \mathrm{Prep}_{\mathbb{Q}}(\psi_{\mathbf{a}}) \leq C_d.
\end{equation*}
\end{conj}

We remark that Morton and Silverman actually conjecture that for any $n \geq 1$, the number of preperiodic points of degree $d$ morphisms $\mathbb{P}^n \to \mathbb{P}^n$ defined over a number field $K$ should be bounded in terms of $d$, $[K : \mathbb{Q}]$ and $n$. In addition, it is worth pointing out that Fakhruddin \cite[Remark~2.6]{MR1995861} has noticed that the particular case $K = \mathbb{Q}$ of this conjecture implies the general prediction.

We note that Conjecture~\ref{Conjecture UBC} is analogous to the celebrated result of Mazur \cite{MR488287} on the number of torsion points on elliptic curves defined over $\mathbb{Q}$. However, its setting critically lacks the structure coming from the group law and therefore remains far out of reach of current technology. Indeed, the best result in the direction of Conjecture~\ref{Conjecture UBC} allows the constant $C_d$ to depend on the number of prime numbers where the polynomial $\psi_{\mathbf{a}} \in \mathscr{P}_d$ does not have potentially good reduction (see Benedetto's work \cite[Main Theorem]{MR2339471}). In the worst possible case, this only yields
\begin{equation}
\label{Upper bound Benedetto}
\# \mathrm{Prep}_{\mathbb{Q}}(\psi_{\mathbf{a}}) \ll 1 + \log \mathscr{H}(\psi_{\mathbf{a}}),
\end{equation}
where the implied constant depends at most on $d$. We point out that throughout the article we use the notation $\ll \cdot$ as a convenient replacement for $=O( \cdot )$. However, we note that for certain weighted homogeneous one parameter families of polynomials, Ingram \cite{MR3988672} has shown that the analog of Conjecture~\ref{Conjecture UBC} holds. We also remark that there has been considerable activity in the case of quadratic polynomials. Most notably, Poonen \cite{MR1617987} has proved that one can take $C_2=8$ in Conjecture~\ref{Conjecture UBC} provided that there does not exist a quadratic polynomial with a rational periodic point of exact period larger than $3$. Unfortunately, it is only known that if such a point exists then its exact period has to be at least $6$ (see \cite{MR1199627} and \cite{MR1480542}). Finally, Looper \cite{MR4270662, Looper2} has shown that Conjecture~\ref{Conjecture UBC} follows from a particular case of Vojta's conjecture (see also \cite{Panraksa}). The interested reader is invited to refer to \cite[Section~$4$]{MR4007163} for a recent survey of our current knowledge on the uniform boundedness conjecture.

Averaging Benedetto's pointwise bound \cite[Main Theorem]{MR2339471}, we deduce that
\begin{equation*}
\frac1{\# \mathscr{P}_d(X)} \sum_{\psi_{\mathbf{a}} \in \mathscr{P}_d(X)} \# \mathrm{Prep}_{\mathbb{Q}}(\psi_{\mathbf{a}}) \ll (\log \log X) \log \log \log X.
\end{equation*}
Our first result offers a substantial improvement upon this upper bound.

\begin{theorem}
\label{Theorem UBC}
Let $d \geq 2$ and $\varepsilon > 0$. Let also $\vartheta_2 = 1/2$ and for $d \geq 3$, let
\begin{equation*}
\vartheta_d = \frac{2(d+1)}{5d+1}.
\end{equation*}
We have
\begin{equation*}
\frac1{\# \mathscr{P}_d(X)} \sum_{\psi_{\mathbf{a}} \in \mathscr{P}_d(X)} \# \mathrm{Prep}_{\mathbb{Q}}(\psi_{\mathbf{a}}) \ll \frac1{X^{\vartheta_d-\varepsilon}},
\end{equation*}
where the implied constant depends at most on $d$ and $\varepsilon$.
\end{theorem}

It follows in particular from Theorem~\ref{Theorem UBC} that when ordered by height, $100 \%$ of the affine conjugacy classes of degree $d$ polynomials defined over $\mathbb{Q}$ do not have any rational preperiodic point. Theorem~\ref{Theorem UBC} can thus be viewed as a strong average version of Conjecture~\ref{Conjecture UBC}.

We remark that in view of the upper bound \eqref{Upper bound Benedetto}, it follows from Theorem~\ref{Theorem UBC} that all higher order moments of the quantity $\# \mathrm{Prep}_{\mathbb{Q}}(\psi_{\mathbf{a}})$ satisfy the same upper bound as the first moment. In addition, we expect that the contribution from the polynomials $\psi_{\mathbf{a}} \in \mathscr{P}_d$ having a rational periodic point of exact period at least $2$ is negligible compared to the contribution from the polynomials having a rational fixed point. As a result, geometry of numbers heuristics lead us to conjecture that there exists $\gamma_d > 0$ such that
\begin{equation}
\label{Estimate conjecture}
\frac1{\# \mathscr{P}_d(X)} \sum_{\psi_{\mathbf{a}} \in \mathscr{P}_d(X)} \# \mathrm{Prep}_{\mathbb{Q}}(\psi_{\mathbf{a}}) \sim \frac{\gamma_d}{X}.
\end{equation}

As it turns out, we can prove much more precise results than Theorem~\ref{Theorem UBC} by studying the growth of canonical heights. Recall that the canonical height $\hat{h}_{\psi_{\mathbf{a}}} : \mathbb{Q} \to \mathbb{R}_{\geq 0}$ of a polynomial $\psi_{\mathbf{a}} \in \mathscr{P}_d$ is defined by
\begin{equation}
\label{Definition global height}
\hat{h}_{\psi_{\mathbf{a}}}(z) = \lim_{n \to \infty} \frac1{d^n} h(\psi_{\mathbf{a}}^n(z)),
\end{equation}
where $h : \mathbb{Q} \to \mathbb{R}_{\geq 0}$ denotes the logarithmic Weil height. We note that for any $z \in \mathbb{Q}$, we have
\begin{equation}
\label{Equality global height}
\hat{h}_{\psi_{\mathbf{a}}}(\psi_{\mathbf{a}}(z)) = d \cdot \hat{h}_{\psi_{\mathbf{a}}}(z).
\end{equation}
Another crucial property of the canonical height $\hat{h}_{\psi_{\mathbf{a}}}$ is that it vanishes exactly at preperiodic points of $\psi_{\mathbf{a}}$, and more generally it measures how far a given rational number is from being a preperiodic point of $\psi_{\mathbf{a}}$.

In the setting of elliptic curves, Lang \cite[page~$92$]{MR518817} has made a conjecture on the minimum canonical height of non-torsion rational points. Lang's conjecture is believed to be buried very deep and is only known to hold in special cases (see for instance the works of Silverman \cite{MR630588}, and Hindry and Silverman \cite{MR948108}). Nevertheless, we note that the first author \cite{MR4029699} has recently established a statistical version of this conjecture for the family of all elliptic curves defined over the field of rational numbers. Silverman \cite[Conjecture~$4$.$98$]{MR2316407} has formulated an intriguing dynamical analog of Lang's conjecture. Unfortunately, as in the setting of elliptic curves this conjecture remains largely open and only partial or conditional results are known (see for example the works of Ingram \cite{MR2504779, MR3988672} and Looper \cite{MR3953120, Looper2}).

Recall that given $\psi_{\mathbf{a}} \in \mathscr{P}_d$ and $u \geq 0$, the set of $z \in \mathbb{Q}$ such that $\hat{h}_{\psi_{\mathbf{a}}}(z) \leq u$ is a set of bounded Weil height and is thus finite. Hence, we can define
\begin{equation}
\label{Definition lambda1}
\lambda_1(\psi_{\mathbf{a}}) = \min \left\{ \hat{h}_{\psi_{\mathbf{a}}}(z) : z \in \mathbb{Q} \smallsetminus \mathrm{Prep}_{\mathbb{Q}}(\psi_{\mathbf{a}}) \right\}.
\end{equation}
In view of \cite[page~$103$]{MR2884382}, we see that the following conjecture implies Silverman's prediction \cite[Conjecture~$4$.$98$]{MR2316407} when restricted to our setting.

\begin{conj}[Silverman]
\label{Conjecture Lang}
Let $d \geq 2$. There exists $c_d > 0$ such that for any $\psi_{\mathbf{a}} \in \mathscr{P}_d$, we have
\begin{equation*}
\lambda_1(\psi_{\mathbf{a}}) > c_d \log \mathscr{H}(\psi_{\mathbf{a}}).
\end{equation*}
\end{conj}

Maybe surprisingly, our next result shows that we can estimate very precisely the quantity $\lambda_1(\psi_{\mathbf{a}})$ for generic polynomials $\psi_{\mathbf{a}} \in \mathscr{P}_d$.

\begin{theorem}
\label{Theorem Lang}
Let $d \geq 2$ and $\varepsilon > 0$. We have 
\begin{equation*}
\lim_{X \to \infty} \frac1{\# \mathscr{P}_d(X)} \cdot
\# \left\{ \psi_{\mathbf{a}} \in \mathscr{P}_d(X) :
\left| \lambda_1(\psi_{\mathbf{a}}) - \frac{\log \mathscr{H}(\psi_{\mathbf{a}})}{d(d-1)} \right| \leq \varepsilon \log \mathscr{H}(\psi_{\mathbf{a}}) \right\} = 1.
\end{equation*}
\end{theorem}

Theorem~\ref{Theorem Lang} implies in particular that when ordered by height, $100 \%$ of the affine conjugacy classes of degree $d$ polynomials defined over $\mathbb{Q}$ satisfy an optimal version of Conjecture~\ref{Conjecture Lang}.

We are actually able to provide a much deeper investigation of the sets of values of canonical heights. Given $\psi_{\mathbf{a}} \in \mathscr{P}_d$, we thus proceed to generalize the definition \eqref{Definition lambda1} of the quantity $\lambda_1(\psi_{\mathbf{a}})$ by introducing the dynamical successive minima of the polynomial $\psi_{\mathbf{a}}$. Given $N \geq 1$, inspired by Baker and DeMarco \cite[page~9]{MR3141413} we say that an element of $\mathbb{Q}^N$ is $\psi_{\mathbf{a}}$-dynamically independent if it does not satisfy any algebraic relation given by a nonzero polynomial $\mathbf{P} \in \overline{\mathbb{Q}}[X_1, \dots, X_N]$ with the property that
\begin{equation}
\label{Property independence}
\forall (\xi_1, \dots, \xi_N) \in \overline{\mathbb{Q}}^N \ \ \ \mathbf{P}(\xi_1, \dots, \xi_N) = 0 \implies \mathbf{P}(\psi_{\mathbf{a}}(\xi_1), \dots, \psi_{\mathbf{a}}(\xi_N)) = 0.
\end{equation}
Motivated by the analogy with the successive minima of a lattice, we introduce the following definition.

\begin{definition}
\label{Definition successive}
Let $d \geq 2$ and $\psi_{\mathbf{a}} \in \mathscr{P}_d$. For $N \geq 1$, the $N$-th dynamical successive minimum $\lambda_N(\psi_{\mathbf{a}})$ of the polynomial $\psi_{\mathbf{a}}$ is
\begin{equation*}
\lambda_N(\psi_{\mathbf{a}}) = \min \left\{ \max_{i \in \{1, \dots, N \}} \hat{h}_{\psi_{\mathbf{a}}}(z_i) : (z_1, \dots, z_N) \in \mathscr{I}_N(\psi_{\mathbf{a}}) \right\},
\end{equation*}
where $\mathscr{I}_N(\psi_{\mathbf{a}})$ denotes the set of $\psi_{\mathbf{a}}$-dynamically independent elements of $\mathbb{Q}^N$.
\end{definition}

It is worth pointing out that the dynamical successive minima of a polynomial are invariant under the action by conjugacy of the affine group. In addition, it is easy to check that $z \in \mathbb{Q}$ is $\psi_{\mathbf{a}}$-dynamically independent if and only if $z \notin \mathrm{Prep}_{\mathbb{Q}}(\psi_{\mathbf{a}})$. We thus see that Definition~\ref{Definition successive} in the case $N=1$ agrees with the definition \eqref{Definition lambda1}.

Given $\psi_{\mathbf{a}} \in \mathscr{P}_d$ and $N \geq 2$, Proposition~\ref{Proposition lambda} provides us with lower and upper bounds for the quantity $\lambda_N(\psi_{\mathbf{a}})$ under a mild assumption on $\psi_{\mathbf{a}}$. However, it should not be surprising that the lower bound is only of interest when $\log N$ is large compared to $\log \mathscr{H}(\psi_{\mathbf{a}})$, and when this is not the case we have to investigate this problem from a statistical point of view. Our methods will prove to be powerful enough to allow us to handle the situation where the integer $N$ grows with the height of the polynomial $\psi_{\mathbf{a}}$, as stated in Proposition~\ref{Proposition lambda statistical}. The full strength of our techniques is demonstrated by Propositions~\ref{Proposition lambda} and \ref{Proposition lambda statistical} but we have decided to defer the statement of these results to Section~\ref{Section higher} in order to avoid introducing too many technical details here. Instead, we present a direct corollary of Theorems~\ref{Theorem UBC} and \ref{Theorem Lang} and Propositions~\ref{Proposition lambda} and \ref{Proposition lambda statistical} which deals with the case where the integer $N$ is fixed.

\begin{theorem}
\label{Theorem N fixed}
Let $d \geq 3$ and $\varepsilon > 0$. Let also $N \geq 3$. We have
\begin{equation*}
\lim_{X \to \infty} \frac1{\# \mathscr{P}_d(X)} \cdot
\# \left\{ \psi_{\mathbf{a}} \in \mathscr{P}_d(X) :
d \left( 1 - \frac{4}{d+3} \right) - \varepsilon < \frac{\lambda_2(\psi_{\mathbf{a}})}{\lambda_1(\psi_{\mathbf{a}})} \leq d + \varepsilon \right\} = 1,
\end{equation*}
and
\begin{equation*}
\lim_{X \to \infty} \frac1{\# \mathscr{P}_d(X)} \cdot
\# \left\{ \psi_{\mathbf{a}} \in \mathscr{P}_d(X) :
1 \leq \frac{\lambda_N(\psi_{\mathbf{a}})}{\lambda_2(\psi_{\mathbf{a}})} \leq 1 + \frac{4}{d-1} + \varepsilon \right\} = 1.
\end{equation*}
\end{theorem}

It is especially enlightening to restrict Theorem~\ref{Theorem N fixed} to the case where $d$ is assumed to be large as both statements then become optimal. More precisely, it follows in particular from Theorem~\ref{Theorem N fixed} that for any $\varepsilon > 0$, there exists $D \geq 3$ such that for any $d \geq D$ and $N \geq 3$, we have
\begin{equation*}
\lim_{X \to \infty} \frac1{\# \mathscr{P}_d(X)} \cdot
\# \left\{ \psi_{\mathbf{a}} \in \mathscr{P}_d(X) :
\left| \frac{\lambda_2(\psi_{\mathbf{a}})}{\lambda_1(\psi_{\mathbf{a}})} - d \right| \leq \varepsilon d \right\} = 1,
\end{equation*}
and
\begin{equation*}
\lim_{X \to \infty} \frac1{\# \mathscr{P}_d(X)} \cdot
\# \left\{ \psi_{\mathbf{a}} \in \mathscr{P}_d(X) :
0 \leq \frac{\lambda_N(\psi_{\mathbf{a}})}{\lambda_2(\psi_{\mathbf{a}})} - 1 \leq \varepsilon \right\} = 1.
\end{equation*}
In other words, if $d$ is somewhat large then when ordered by height, $100 \%$ of the affine conjugacy classes of degree $d$ polynomials defined over $\mathbb{Q}$ have their second dynamical successive minimum which is about $d$ times larger than their first, while an arbitrarily large number of the next successive minima essentially have the same size.

We now proceed to give a quick sketch of the proof of our results. To establish Theorems~\ref{Theorem UBC} and \ref{Theorem Lang}, as well as the lower bound in Proposition~\ref{Proposition lambda}, and Proposition~\ref{Proposition lambda statistical}, we make key use of the decomposition of the canonical height into local canonical heights, along with the fact that infinity is a superattracting fixed point of any polynomial. Broadly speaking, this allows us to show that most of the time if a polynomial $\psi_{\mathbf{a}} \in \mathscr{P}_d$ has a rational preperiodic point or has an excessively small first dynamical successive minimum, then the integer $a_0$ has to be far from being $k$-free, where $k \geq 2$ is a suitable integer. If there were no restrictions on the choice of $k$, then in Theorem~\ref{Theorem UBC} one could in theory obtain the essentially optimal saving $X^{1-\varepsilon}$, for any $\varepsilon > 0$. Unfortunately, the decomposition of the canonical height into local canonical heights introduces error terms which reflect the fact that the dynamics of a polynomial $\psi_{\mathbf{a}} \in \mathscr{P}_d$ is hard to analyze whenever the prime factorizations of the integers $a_0$ and $a_d$ are intimately linked. Moreover, these error terms become more and more problematic as the integer $k$ grows, which ultimately prevents us from getting closer to the conjectural estimate \eqref{Estimate conjecture}.

Finally, it is worth stressing that along the proof of the upper bound in Proposition~\ref{Proposition lambda} we crucially appeal to the celebrated work of Medvedev and Scanlon \cite{MR3126567}. More precisely, we use the fact that if $\psi_{\mathbf{a}} \in \mathscr{P}_d$ does not belong to a short list of exceptional cases, then the only irreducible algebraic subsets of $\mathbb{A}^N(\overline{\mathbb{Q}})$ defined by polynomials satisfying the property \eqref{Property independence} are those which are defined by polynomials $\mathbf{P} \in \overline{\mathbb{Q}}[X_1, \dots, X_N]$ of the shape
\begin{equation*}
\mathbf{P}(X_1, \dots, X_N) = X_i - g(X_j),
\end{equation*}
for some $i, j \in \{1, \dots, N\}$ and $g \in \overline{\mathbb{Q}}[z]$ commuting with $\psi_{\mathbf{a}}$. For polynomials $\psi_{\mathbf{a}} \in \mathscr{P}_d$ satisfying a mild assumption this allows us to exhibit many $\psi_{\mathbf{a}}$-dynamically independent vectors of rational numbers of controlled Weil height, and we can thus conclude by using a classical upper bound for the difference between the canonical height with respect to $\psi_{\mathbf{a}}$ and the Weil height.

It may be useful to record here that Theorems~\ref{Theorem UBC} and \ref{Theorem Lang} are respectively established in Sections~\ref{Section UBC} and \ref{Section Lang}, while Propositions~\ref{Proposition lambda} and \ref{Proposition lambda statistical} and Theorem~\ref{Theorem N fixed} are proved in Section~\ref{Section higher}.

We finish this introduction by mentioning that we believe that our techniques are robust enough to prove analogous results for general number fields. However, there is no doubt that generalizing certain of the tools that we use would have made our proofs much longer and much more intricate. For the sake of concision and clarity we have thus decided not to follow this path.

\subsection*{Acknowledgements}

The research of the first-named author is integrally funded by the Swiss National Science Foundation through the SNSF Professorship number $170565$ awarded to the project \textit{Height of rational points on algebraic varieties}. The research of the second-named author was also funded by this grant during the period that she spent at the University of Basel. Both the financial support of the SNSF and the perfect working conditions provided by the University of Basel are gratefully acknowledged. It is a pleasure for the authors to thank Laura DeMarco, Dragos Ghioca, Fabien Pazuki and Joe Silverman for their interest and for several insightful comments on an earlier version of this manuscript.

\section{Preliminaries}

\subsection{Inequalities involving the local canonical heights}

We start by introducing some classical notation. Given a prime number $p$, we let $|\cdot|_p$ denote the usual $p$-adic absolute value. We thus have $|0|_p=0$ and if we let $v_p(z)$ be the $p$-adic valuation of $z \in \mathbb{Q}^{\times}$ then we have
\begin{equation*}
|z|_p = \frac1{p^{v_p(z)}}.
\end{equation*}

Let $d \geq 2$ and $\psi_{\mathbf{a}} \in \mathscr{P}_d$. Letting $v$ denote either a prime number $p$ or $\infty$, we recall that the local canonical heights $\hat{\lambda}_{\psi_{\mathbf{a}},v} : \mathbb{Q} \to \mathbb{R}_{\geq 0}$ of the polynomial $\psi_{\mathbf{a}}$ are defined by setting
\begin{equation}
\label{Definition local}
\hat{\lambda}_{\psi_{\mathbf{a}},v}(z) = \lim_{n \to \infty} \frac1{d^n} \log \max \left\{ |\psi_{\mathbf{a}}^n(z)|_v, 1 \right\}.
\end{equation}
We will make frequent use of the fact that for any $z \in \mathbb{Q}$ we have $\hat{\lambda}_{\psi_{\mathbf{a}},v}(z) \geq 0$, and it will also be useful to note that
\begin{equation}
\label{Equality local height}
\hat{\lambda}_{\psi_{\mathbf{a}},v}(\psi_{\mathbf{a}}(z)) = d \cdot \hat{\lambda}_{\psi_{\mathbf{a}},v}(z).
\end{equation}
Moreover, it follows from the product formula that
\begin{equation}
\label{Equality product formula}
\hat{h}_{\psi_{\mathbf{a}}} = \sum_p \hat{\lambda}_{\psi_{\mathbf{a}},p} + \hat{\lambda}_{\psi_{\mathbf{a}},\infty}.
\end{equation}
In addition, we set
\begin{equation}
\label{Definition r}
r_v(\psi_{\mathbf{a}}) = \max \left\{ \left|\frac{a_0}{a_d}\right|_v^{1/(d-1)}, \max_{i \in \{0, \dots, d-2 \}} \left| \frac{a_i}{a_d} \right|_v^{1/(d-i)} \right\}.
\end{equation}
We note here that we simply write $| \cdot |$ to denote the Archimedean absolute value $| \cdot |_{\infty}$.

The three following results are classical but we include their proofs for completeness. They crucially rely on the fact that infinity is a superattracting fixed point of any polynomial. Once combined, they imply in particular that given a polynomial $\psi_{\mathbf{a}} \in \mathscr{P}_d$, most rational numbers have large canonical height with respect to $\psi_{\mathbf{a}}$.

\begin{lemma}
\label{Lemma non Archimedean equality}
Let $d \geq 2$ and $\psi_{\mathbf{a}} \in \mathscr{P}_d$. Let also $p$ be a prime number. If $z \in \mathbb{Q}^{\times}$ satisfies $|z|_p>r_p(\psi_{\mathbf{a}})$ then
\begin{equation*}
\hat{\lambda}_{\psi_{\mathbf{a}},p}(z) - \log |z|_p = \frac1{d-1} \log \left|\frac{a_d}{a_0}\right|_p.
\end{equation*}
\end{lemma}

\begin{proof}
The assumption $|z|_p > r_p(\psi_{\mathbf{a}})$ implies that for any $i \in \{0, \dots, d-2 \}$ we have
\begin{equation*}
\left|\frac{a_d}{a_0} z^d\right|_p > \left|\frac{a_i}{a_0} z^i \right|_p.
\end{equation*}
Hence, recalling the definition \eqref{Definition psi} of the polynomial $\psi_{\mathbf{a}}$ we observe that
\begin{equation*}
|\psi_{\mathbf{a}}(z)|_p = \left|\frac{a_d}{a_0} z^d\right|_p.
\end{equation*}
Therefore, it follows from the assumption $|z|_p > r_p(\psi_{\mathbf{a}})$ that $|\psi_{\mathbf{a}}(z)|_p \geq |z|_p$, and an elementary induction thus shows that for any $n \geq 1$, we have 
\begin{equation*}
|\psi_{\mathbf{a}}^n(z)|_p = \left|\frac{a_d}{a_0}\right|_p^{(d^n-1)/(d-1)} |z|_p^{d^n}.
\end{equation*}
In addition, the assumption $|z|_p>r_p(\psi_{\mathbf{a}})$ also yields
\begin{equation*}
\left|\frac{a_d}{a_0}\right|_p^{1/(d-1)} |z|_p > 1.
\end{equation*}
As a result, we conclude that if $n$ is large enough then $|\psi_{\mathbf{a}}^n(z)|_p >1$. Recalling the definition \eqref{Definition local} of the local canonical height $\hat{\lambda}_{\psi_{\mathbf{a}},p}$, we see that this completes the proof.
\end{proof}

The following statement is a straightforward consequence of Lemma~\ref{Lemma non Archimedean equality}.

\begin{lemma}
\label{Lemma non Archimedean}
Let $d \geq 2$ and $\psi_{\mathbf{a}} \in \mathscr{P}_d$. Let also $p$ be a prime number. For any $z \in \mathbb{Q}^{\times}$, we have
\begin{equation*}
\hat{\lambda}_{\psi_{\mathbf{a}},p}(z) - \log |z|_p \geq - \log r_p(\psi_{\mathbf{a}}).
\end{equation*}
\end{lemma}

\begin{proof}
We start by noting that the statement is clear when $|z|_p \leq r_p(\psi_{\mathbf{a}})$. Moreover, when $|z|_p>r_p(\psi_{\mathbf{a}})$ it directly follows from Lemma~\ref{Lemma non Archimedean equality} and the inequality
\begin{equation*}
\frac1{d-1} \log \left|\frac{a_d}{a_0}\right|_p \geq -\log r_p(\psi_{\mathbf{a}}),
\end{equation*}
which completes the proof.
\end{proof}

Our next result is the Archimedean analog of Lemma~\ref{Lemma non Archimedean}.

\begin{lemma}
\label{Lemma Archimedean}
Let $d \geq 2$ and $\psi_{\mathbf{a}} \in \mathscr{P}_d$. For any $z \in \mathbb{Q}^{\times}$, we have
\begin{equation*}
\hat{\lambda}_{\psi_{\mathbf{a}},\infty}(z) - \log|z| \geq - \log r_{\infty}(\psi_{\mathbf{a}}) - \log d.
\end{equation*}
\end{lemma}

\begin{proof}
We first note that the statement is clear when $|z| \leq d \cdot r_{\infty}(\psi_{\mathbf{a}})$. We now assume that $|z| > d \cdot r_{\infty}(\psi_{\mathbf{a}})$ and we follow closely the lines of the proof of Lemma~\ref{Lemma non Archimedean equality} to show that
\begin{equation}
\label{Goal Archimedean}
\hat{\lambda}_{\psi_{\mathbf{a}},\infty}(z) - \log|z| \geq \frac1{d-1} \log \left| \frac{a_d}{a_0} \right| - \frac1{d-1} \log 2.
\end{equation}
The assumption $|z|> d \cdot r_{\infty}(\psi_{\mathbf{a}})$ implies that for any $i \in \{0, \dots, d-2 \}$ we have
\begin{equation*}
\left|\frac{a_d}{a_0} z^d\right| > d^{d-i} \left|\frac{a_i}{a_0} z^i \right|.
\end{equation*}
As a result, recalling the definition \eqref{Definition psi} of the polynomial $\psi_{\mathbf{a}}$ we see that the triangle inequality yields in particular
\begin{equation*}
|\psi_{\mathbf{a}}(z)| \geq \left|\frac{a_d}{2a_0} z^d \right|.
\end{equation*}
Hence, using the assumption $|z| > d \cdot r_{\infty}(\psi_{\mathbf{a}})$ we see that $|\psi_{\mathbf{a}}(z)|_p \geq |z|_p$, and an elementary induction thus shows that for any $n \geq 1$, we have
\begin{equation*}
|\psi_{\mathbf{a}}^n(z)| \geq \left| \frac{a_d}{2a_0} \right|^{(d^n-1)/(d-1)} |z|^{d^n}.
\end{equation*}
Moreover, the assumption $|z|> d \cdot r_{\infty}(\psi_{\mathbf{a}})$ also gives
\begin{equation*}
\left| \frac{a_d}{d a_0} \right|^{1/(d-1)} |z| > 1.
\end{equation*}
Therefore, we deduce that if $n$ is large enough then $|\psi_{\mathbf{a}}^n(z)| >1$. Recalling the definition \eqref{Definition local} of the local canonical height $\hat{\lambda}_{\psi_{\mathbf{a}},\infty}$, we see that the lower bound \eqref{Goal Archimedean} follows. We finally remark that the inequality
\begin{equation*}
\frac1{d-1} \log \left| \frac{a_d}{a_0} \right| \geq - \log r_{\infty}(\psi_{\mathbf{a}})
\end{equation*}
allows us to complete the proof.
\end{proof}

\subsection{Inequalities involving the global canonical height}

For $k \geq 2$ and $n \geq 1$, we define the $k$-free part $\mathrm{s}_k(n)$ of the integer $n$ by
\begin{equation*}
\mathrm{s}_k(n) = \prod_{\substack{p \mid n \\ p^k \nmid n}} \frac1{|n|_p}.
\end{equation*}
Throughout the article we say that $x \in \mathbb{Z}$ and $y \geq 1$ are respectively the numerator and the denominator of a given $z \in \mathbb{Q}$ if and only if we have $z = x/y$ and $\gcd(x,y)=1$.

Our next result will be the most important tool in the proof of Theorems~\ref{Theorem UBC} and \ref{Theorem Lang} in the case $d=2$.

\begin{lemma}
\label{Lemma d=2}
Let $\psi_{\mathbf{a}} \in \mathscr{P}_2$. For any $z \in \mathbb{Q}$, we have
\begin{equation*}
\hat{h}_{\psi_{\mathbf{a}}}(z) \geq \frac1{2} \log \mathrm{s}_2(a_0).
\end{equation*}
\end{lemma}

\begin{proof}
Let $x \in \mathbb{Z}$ denote the numerator of $z$. We start by writing
\begin{equation}
\label{Equality sa0 2}
\log \mathrm{s}_2(a_0) = \sum_{\substack{p \mid \mathrm{s}_2(a_0) \\ p \nmid x}} \log p + \sum_{\substack{p \mid \mathrm{s}_2(a_0) \\ p \mid x}} \log p.
\end{equation}
We note that given a prime number $p$ dividing $\mathrm{s}_2(a_0)$, we have $p \nmid a_2$ since $(a_2,a_0) \in \mathbb{Z}_{\mathrm{prim}}^2$ so it follows that
\begin{equation*}
r_p(\psi_{\mathbf{a}}) = \frac1{p^{1/2}}.
\end{equation*}
Therefore, for each prime number $p$ dividing $\mathrm{s}_2(a_0)$ and such that $p \nmid x$ we are in position to apply Lemma~\ref{Lemma non Archimedean equality}. We deduce in particular that
\begin{equation}
\label{Upper bound nmid x}
\sum_{\substack{p \mid \mathrm{s}_2(a_0) \\ p \nmid x}} \log p \leq \sum_{\substack{p \mid \mathrm{s}_2(a_0) \\ p \nmid x}} \hat{\lambda}_{\psi_{\mathbf{a}},p}(z).
\end{equation}

Next, given any prime number $p$ dividing $\mathrm{s}_2(a_0)$ and $x$, it follows from the definition \eqref{Definition psi} of the polynomial $\psi_{\mathbf{a}}$ that $|\psi_{\mathbf{a}}(z)|_p = 1$. Applying Lemma~\ref{Lemma non Archimedean equality} for each prime number $p$ dividing $\mathrm{s}_2(a_0)$ and $x$ we thus obtain
\begin{equation*}
\sum_{\substack{p \mid \mathrm{s}_2(a_0) \\ p \mid x}} \log p = \sum_{\substack{p \mid \mathrm{s}_2(a_0) \\ p \mid x}} \hat{\lambda}_{\psi_{\mathbf{a}},p}(\psi_{\mathbf{a}}(z)).
\end{equation*}
Hence, the equality \eqref{Equality local height} yields
\begin{equation}
\label{Equality mid x}
\sum_{\substack{p \mid \mathrm{s}_2(a_0) \\ p \mid x}} \log p = 2 \sum_{\substack{p \mid \mathrm{s}_2(a_0) \\ p \mid x}} \hat{\lambda}_{\psi_{\mathbf{a}},p}(z).
\end{equation}

Putting together the equalities \eqref{Equality sa0 2} and \eqref{Equality mid x} and the upper bound \eqref{Upper bound nmid x} we get
\begin{equation*}
\log \mathrm{s}_2(a_0) \leq 2 \sum_{p \mid \mathrm{s}_2(a_0)} \hat{\lambda}_{\psi_{\mathbf{a}},p}(z).
\end{equation*}
We immediately complete the proof by appealing to the equality \eqref{Equality product formula}.
\end{proof}

We now introduce an arithmetic quantity which will play a pivotal role in our work. Given $\ell, m \in \mathbb{Z}$, we let $\Delta_{\ell}(m)$ denote the largest divisor of $m$ whose radical divides $\ell$, that is
\begin{equation}
\label{Definition Delta}
\Delta_{\ell}(m) = \prod_{p \mid \ell} \frac1{|m|_p}.
\end{equation}
The remainder of this section deals with the case $d \geq 3$. Our next result asserts that for a generic polynomial $\psi_{\mathbf{a}} \in \mathscr{P}_d$, a typical rational number which has small canonical height with respect to $\psi_{\mathbf{a}}$ also has small Weil height.

\begin{lemma}
\label{Lemma Weil}
Let $d \geq 3$ and $\psi_{\mathbf{a}} \in \mathscr{P}_d$. Let also $z \in \mathbb{Q}^{\times}$ and let $x \in \mathbb{Z}$ and $y \geq 1$ be respectively the numerator and the denominator of $z$. We have
\begin{equation*}
\hat{h}_{\psi_{\mathbf{a}}}(z) - \log y \geq - \frac1{2} \log \Delta_y(a_d),
\end{equation*}
and
\begin{equation*}
\hat{h}_{\psi_{\mathbf{a}}}(z) - \log |x| \geq - \frac1{2} \log \Delta_y(a_d) - \log r_{\infty}(\psi_{\mathbf{a}}) - \log d.
\end{equation*}
\end{lemma}

\begin{proof}
We start by noting that since $a_i \in \mathbb{Z}$ for any $i \in \{0, \dots, d-2 \}$, the assumption $d \geq 3$ implies that for any prime number $p$, we have
\begin{equation}
\label{Upper bound rp}
r_p(\psi_{\mathbf{a}}) \leq \frac1{|a_d|_p^{1/2}}.
\end{equation}
Therefore, applying Lemma~\ref{Lemma non Archimedean} for each prime number $p$ dividing $y$ we obtain
\begin{equation*}
\sum_{p \mid y} \hat{\lambda}_{\psi_{\mathbf{a}},p}(z) - \log y \geq - \frac1{2} \sum_{p \mid y} \log \frac1{|a_d|_p}.
\end{equation*}
Appealing to the equality \eqref{Equality product formula} we thus deduce that
\begin{equation}
\label{Inequality y}
\hat{h}_{\psi_{\mathbf{a}}}(z) - \hat{\lambda}_{\psi_{\mathbf{a}},\infty}(z) - \log y \geq - \frac1{2} \log \Delta_{y}(a_d),
\end{equation}
which implies the first inequality claimed.

Next, we note that Lemma~\ref{Lemma Archimedean} states that
\begin{equation}
\label{Inequality z}
\hat{\lambda}_{\psi_{\mathbf{a}},\infty}(z) - \left( \log|x| - \log y \right) \geq - \log r_{\infty}(\psi_{\mathbf{a}}) - \log d.
\end{equation}
Combining the lower bounds \eqref{Inequality y} and \eqref{Inequality z} we obtain the second inequality claimed, which completes the proof.
\end{proof}

We now record an immediate consequence of Lemma~\ref{Lemma Weil} which will be used in the proof of the lower bound in Proposition~\ref{Proposition lambda}. Given $\psi_{\mathbf{a}} \in \mathscr{P}_d$, it provides us with a lower bound for the difference between the canonical height with respect to $\psi_{\mathbf{a}}$ and the Weil height.

\begin{lemma}
\label{Lemma lower bound}
Let $d \geq 3$ and $\psi_{\mathbf{a}} \in \mathscr{P}_d$. For any $z \in \mathbb{Q}$, we have
\begin{equation*}
\hat{h}_{\psi_{\mathbf{a}}}(z) - h(z) \geq - \frac1{2} \log \mathscr{H}(\psi_{\mathbf{a}}) - \log d.
\end{equation*}
\end{lemma}

\begin{proof}
We start by noting that the statement clearly holds in the case $z=0$ so we now assume that $z \in \mathbb{Q}^{\times}$. Let $y \geq 1$ denote the denominator of $z$. Lemma~\ref{Lemma Weil} shows in particular that
\begin{equation*}
\hat{h}_{\psi_{\mathbf{a}}}(z) - h(z) \geq - \frac1{2} \log \Delta_y(a_d) - \log r_{\infty}(\psi_{\mathbf{a}}) - \log d.
\end{equation*}
But it is clear that $\Delta_y(a_d) \leq |a_d|$, and moreover the assumption $d \geq 3$ implies that
\begin{equation*}
r_{\infty}(\psi_{\mathbf{a}}) \leq \frac{\mathscr{H}(\psi_{\mathbf{a}})^{1/2}}{|a_d|^{1/2}}.
\end{equation*}
We thus have
\begin{equation*}
\frac1{2} \log \Delta_y(a_d) + \log r_{\infty}(\psi_{\mathbf{a}}) \leq \frac1{2} \log \mathscr{H}(\psi_{\mathbf{a}}),
\end{equation*}
which completes the proof.
\end{proof}

Given $\psi_{\mathbf{a}} \in \mathscr{P}_d$ and $x,y \in \mathbb{Z}$, it is convenient to set
\begin{equation}
\label{Definition sigma}
\sigma_{\mathbf{a}}(x,y) = a_d x^d + a_{d-2} x^{d-2} y^2 + \cdots + a_1 x y^{d-1}.
\end{equation}
The following result deals with the case $d \geq 3$ and will be the key tool in the proof of Theorems~\ref{Theorem UBC} and \ref{Theorem Lang}, and Proposition~\ref{Proposition lambda statistical}. It asserts that given a generic polynomial $\psi_{\mathbf{a}} \in \mathscr{P}_d$ and $k \in \{2, 3\}$, if the canonical height with respect to $\psi_{\mathbf{a}}$ of $z \in \mathbb{Q}^{\times}$ is small then a large factor of the $k$-free part of $a_0$ has to divide $\sigma_{\mathbf{a}}(x,y)$. We will make use of the case $k=2$ in the proof of Theorem~\ref{Theorem Lang} and Proposition~\ref{Proposition lambda statistical}. We note that the exponent stated in Theorem~\ref{Theorem UBC} crucially relies upon the case $k=3$ but the case $k=2$ is actually sufficient to establish a version of Theorem~\ref{Theorem UBC} with a weaker exponent.

\begin{lemma}
\label{Lemma a0}
Let $d \geq 3$ and $k \in \{2, 3\}$. Let also $\psi_{\mathbf{a}} \in \mathscr{P}_d$. Finally, let $z \in \mathbb{Q}^{\times}$ and let $x \in \mathbb{Z}$ and $y \geq 1$ be respectively the numerator and the denominator of $z$. We have
\begin{align*}
d(d-1) \hat{h}_{\psi_{\mathbf{a}}}(z) \geq & \ \frac{d+k-2}{k-1} \log \frac{\mathrm{s}_k(a_0)}{\gcd(\mathrm{s}_k(a_0), \sigma_{\mathbf{a}}(x,y))} - \log \gcd (\mathrm{s}_k(a_0), a_d) \\
& - \frac{d-1}{2} \log \Delta_{\mathrm{s}_k(a_0)}(a_d).
\end{align*}
\end{lemma}

\begin{proof}
Let $w \geq 1$ denote the denominator of $\psi_{\mathbf{a}}(z)$. We start by writing
\begin{equation}
\label{Equality sa0}
\log \mathrm{s}_k(a_0) = \sum_{\substack{p \mid \mathrm{s}_k(a_0) \\ p \mid w}} \log \frac1{|a_0|_p} + \sum_{\substack{p \mid \mathrm{s}_k(a_0) \\ p \nmid w}} \log \frac1{|a_0|_p}.
\end{equation}
For any prime $p$, the $p$-adic valuation of $\mathrm{s}_k(a_0)$ is at most $k-1$ so we clearly have the inequality
\begin{equation*}
\sum_{\substack{p \mid \mathrm{s}_k(a_0) \\ p \mid w}} \log \frac1{|a_0|_p} \leq (k-1) \sum_{\substack{p \mid \mathrm{s}_k(a_0) \\ p \mid w}} \log |\psi_{\mathbf{a}}(z)|_p.
\end{equation*}
We now remark that the upper bound \eqref{Upper bound rp} shows that if $p$ is a prime number dividing $w$ but such that $p^2 \nmid a_d$ then we are in position to use Lemma~\ref{Lemma non Archimedean equality}. Therefore, for each prime number $p$ dividing $\mathrm{s}_k(a_0)$ and $w$ we proceed to apply either Lemma~\ref{Lemma non Archimedean equality} if $p^2 \nmid a_d$ or Lemma~\ref{Lemma non Archimedean} if $p^2 \mid a_d$. Using the upper bound \eqref{Upper bound rp}, we deduce in particular that
\begin{align}
\begin{split}
\label{Upper bound after lemmas}
\sum_{\substack{p \mid \mathrm{s}_k(a_0) \\ p \mid w}} \log \frac1{|a_0|_p} \leq & \ (k-1) \sum_{p \mid \mathrm{s}_k(a_0)} \hat{\lambda}_{\psi_{\mathbf{a}},p}(\psi_{\mathbf{a}}(z)) - \frac{k-1}{d-1} \sum_{\substack{p \mid \gcd(\mathrm{s}_k(a_0), w) \\ p^2 \nmid a_d}} \log \left|\frac{a_d}{a_0}\right|_p \\
& + \frac{k-1}{2} \sum_{\substack{p \mid \mathrm{s}_k(a_0) \\ p^2 \mid a_d}} \log \frac1{|a_d|_p}.
\end{split}
\end{align}
Moreover, we note that
\begin{align}
\nonumber
\sum_{\substack{p \mid \gcd(\mathrm{s}_k(a_0),w) \\ p^2 \nmid a_d}} \log \left|\frac{a_d}{a_0}\right|_p & \geq \sum_{\substack{p \mid \gcd(\mathrm{s}_k(a_0),w) \\ p \nmid a_d}} \log \frac1{|a_0|_p} \\
\nonumber
& \geq \sum_{\substack{p \mid \mathrm{s}_k(a_0) \\ p \nmid a_d}} \log \frac1{|a_0|_p} - \sum_{\substack{p \mid \mathrm{s}_k(a_0) \\ p \nmid a_dw}} \log \frac1{|a_0|_p} \\
\label{Lower bound refinement}
& \geq \log \mathrm{s}_k(a_0) - \sum_{\substack{p \mid \mathrm{s}_k(a_0) \\ p \mid a_d}} \log \frac1{|a_0|_p} - \sum_{\substack{p \mid \mathrm{s}_k(a_0) \\ p \nmid w}} \log \frac1{|a_0|_p}.
\end{align}
Putting together the equality \eqref{Equality sa0} and the inequalities \eqref{Upper bound after lemmas} and \eqref{Lower bound refinement}, we derive
\begin{align}
\begin{split}
\label{Lower bound intermediate}
(k-1) \sum_{p \mid \mathrm{s}_k(a_0)} \hat{\lambda}_{\psi_{\mathbf{a}},p}(\psi_{\mathbf{a}}(z)) \geq & \ 
\frac{d+k-2}{d-1} \log \mathrm{s}_k(a_0) - \frac{d+k-2}{d-1} \sum_{\substack{p \mid \mathrm{s}_k(a_0) \\ p \nmid w}} \log \frac1{|a_0|_p} \\
& - \frac{k-1}{d-1} \sum_{\substack{p \mid \mathrm{s}_k(a_0) \\ p \mid a_d}} \log \frac1{|a_0|_p} - \frac{k-1}{2} \sum_{\substack{p \mid \mathrm{s}_k(a_0) \\ p^2 \mid a_d}} \log \frac1{|a_d|_p}.
\end{split}
\end{align}

Next, using the equalities \eqref{Equality product formula} and \eqref{Equality global height} we get
\begin{equation}
\label{Upper bound backwards}
\sum_{p \mid \mathrm{s}_k(a_0)} \hat{\lambda}_{\psi_{\mathbf{a}},p}(\psi_{\mathbf{a}}(z)) \leq d \cdot \hat{h}_{\psi_{\mathbf{a}}}(z).
\end{equation}
Furthermore, recalling the definition \eqref{Definition sigma} of the quantity $\sigma_{\mathbf{a}}(x,y)$, we see that it follows from the definition \eqref{Definition psi} of the polynomial $\psi_{\mathbf{a}}$ that
\begin{equation*}
w = \frac{a_0y^d}{\gcd(a_0y^d, \sigma_{\mathbf{a}}(x,y))}.
\end{equation*}
We thus deduce that
\begin{equation}
\label{Upper bound nmid w}
\sum_{\substack{p \mid \mathrm{s}_k(a_0) \\ p \nmid w}} \log \frac1{|a_0|_p} \leq \log \gcd(\mathrm{s}_k(a_0), \sigma_{\mathbf{a}}(x,y)).
\end{equation}
Putting together the lower bound \eqref{Lower bound intermediate} and the upper bounds \eqref{Upper bound backwards} and \eqref{Upper bound nmid w}, we derive
\begin{align*}
d (d-1) \hat{h}_{\psi_{\mathbf{a}}}(z) \geq & \ 
\frac{d+k-2}{k-1} \log \frac{\mathrm{s}_k(a_0)}{\gcd(\mathrm{s}_k(a_0), \sigma_{\mathbf{a}}(x,y))} - \sum_{\substack{p \mid \mathrm{s}_k(a_0) \\ p \mid a_d}} \log \frac1{|a_0|_p} \\
& - \frac{d-1}{2} \sum_{\substack{p \mid \mathrm{s}_k(a_0) \\ p^2 \mid a_d}} \log \frac1{|a_d|_p}.
\end{align*}

In the case $k=2$, we note that
\begin{equation*}
\sum_{\substack{p \mid \mathrm{s}_k(a_0) \\ p \mid a_d}} \log \frac1{|a_0|_p} \leq \log \gcd (\mathrm{s}_k(a_0), a_d).
\end{equation*}
Since we clearly have
\begin{equation*}
\sum_{\substack{p \mid \mathrm{s}_k(a_0) \\ p^2 \mid a_d}} \log \frac1{|a_d|_p} \leq \log \Delta_{\mathrm{s}_k(a_0)}(a_d),
\end{equation*}
we see that this completes the proof in the case $k=2$.

In the case $k=3$, we remark that
\begin{equation*}
\sum_{\substack{p \mid \mathrm{s}_k(a_0) \\ p \mid a_d}} \log \frac1{|a_0|_p} = \log \gcd (\mathrm{s}_k(a_0), a_d) + \sum_{\substack{p^2 \mid \mathrm{s}_k(a_0) \\ p^2 \nmid a_d}} \log \frac1{|a_d|_p}.
\end{equation*}
Hence, it follows from our assumption $d \geq 3$ that
\begin{equation*}
\sum_{\substack{p \mid \mathrm{s}_k(a_0) \\ p \mid a_d}} \log \frac1{|a_0|_p} + \frac{d-1}{2} \sum_{\substack{p \mid \mathrm{s}_k(a_0) \\ p^2 \mid a_d}} \log \frac1{|a_d|_p} \leq \log \gcd (\mathrm{s}_k(a_0), a_d) + \frac{d-1}{2} \log \Delta_{\mathrm{s}_k(a_0)}(a_d),
\end{equation*}
which completes the proof in the case $k = 3$.
\end{proof}

In order to establish Theorem~\ref{Theorem Lang} and the upper bound in Proposition~\ref{Proposition lambda} we will also need the following result which, given $\psi_{\mathbf{a}} \in \mathscr{P}_d$, provides us with an upper bound for the difference between the canonical height with respect to $\psi_{\mathbf{a}}$ and the Weil height.

\begin{lemma}
\label{Lemma upper bound}
Let $d \geq 2$ and $\psi_{\mathbf{a}} \in \mathscr{P}_d$. For any $z \in \mathbb{Q}$, we have
\begin{equation*}
\hat{h}_{\psi_{\mathbf{a}}}(z) - h(z) \leq \frac1{d-1} \log \mathscr{H}(\psi_{\mathbf{a}}) + \frac{\log d}{d-1}.
\end{equation*}
\end{lemma}

\begin{proof}
For any $n \geq 0$ we let $x_n \in \mathbb{Z}$ and $y_n \geq 1$ be respectively the numerator and the denominator of $\psi_{\mathbf{a}}^n(z)$. Given $x,y \in \mathbb{Z}$, recall the definition \eqref{Definition sigma} of the quantity $\sigma_{\mathbf{a}}(x,y)$. It follows from the definition \eqref{Definition psi} of the polynomial $\psi_{\mathbf{a}}$ that for any $n \geq 0$, we have
\begin{equation*}
\psi_{\mathbf{a}}^{n+1}(z) = \frac{\sigma_{\mathbf{a}}(x_n,y_n) + a_0 y_n^d}{a_0 y_n^d}.
\end{equation*}
We thus have
\begin{equation*}
h(\psi_{\mathbf{a}}^{n+1}(z)) \leq \log \max \left\{ \left| \sigma_{\mathbf{a}}(x_n,y_n) + a_0 y_n^d \right|, a_0 y_n^d \right\}.
\end{equation*}
But it is clear that
\begin{equation*}
\left| \sigma_{\mathbf{a}}(x_n,y_n) + a_0 y_n^d \right| \leq d \mathscr{H}(\psi_{\mathbf{a}}) \max \left\{ |x_n|^d, y_n^d \right\}.
\end{equation*}
As a result, we see that
\begin{equation*}
h(\psi_{\mathbf{a}}^{n+1}(z)) \leq \log (d \mathscr{H}(\psi_{\mathbf{a}})) + d \cdot h(\psi_{\mathbf{a}}^n(z)).
\end{equation*}
Hence, we derive
\begin{equation*}
\frac1{d^{n+1}} h(\psi_{\mathbf{a}}^{n+1}(z)) - \frac1{d^n} h(\psi_{\mathbf{a}}^n(z)) \leq \frac1{d^{n+1}} \log (d \mathscr{H}(\psi_{\mathbf{a}})).
\end{equation*}
Summing this inequality over the integer $n$, we deduce that for any $m \geq 0$, we have
\begin{equation*}
\frac1{d^m} h(\psi_{\mathbf{a}}^m(z)) - h(z) \leq \frac1{d^m} \cdot \frac{d^m-1}{d-1} \log (d \mathscr{H}(\psi_{\mathbf{a}})).
\end{equation*}
Recalling the definition \eqref{Definition global height} of the canonical height and letting $m$ tend to $\infty$, we see that this completes the proof.
\end{proof}

\subsection{A subfamily of the family of all affine conjugacy classes of polynomials}

It will be useful to remark that a M\"{o}bius inversion shows that
\begin{equation}
\label{Estimate P}
\# \mathscr{P}_d(X) = \frac{2^{d-1}}{\zeta(d)} X^d \left( 1 + O \left( \frac{\log X}{X} \right) \right).
\end{equation}

For $n \geq 1$ we let $\rad(n)$ denote the radical of the integer $n$, that is
\begin{equation*}
\rad(n) = \prod_{p \mid n} p.
\end{equation*}
It was proved by de Bruijn \cite[Theorem~1]{MR147461} that
\begin{equation*}
\log \left( \sum_{n \leq X} \frac1{\rad(n)} \right) \sim \left( \frac{8 \log X}{\log \log X} \right)^{1/2}.
\end{equation*}
The following result is an immediate consequence of this estimate.

\begin{lemma}
\label{Lemma dB}
Let $\varepsilon > 0$. We have
\begin{equation*}
\sum_{n \leq X} \frac1{\rad(n)} \ll X^{\varepsilon},
\end{equation*}
where the implied constant depends at most on $\varepsilon$.
\end{lemma}

Given $\ell, m \in \mathbb{Z}$, recall the definition \eqref{Definition Delta} of the largest divisor $\Delta_{\ell}(m)$ of $m$ whose radical divides $\ell$. For $\delta \in (0,1)$, we introduce the set
\begin{equation}
\label{Definition Pddelta}
\mathscr{P}_d^{(\delta)}(X) = \left\{ \psi_{\mathbf{a}} \in \mathscr{P}_d(X) : \Delta_{a_0}(a_d) \leq X^{\delta} \right\}.
\end{equation}
Our next task is to prove that for any $\delta \in (0,1)$, the cardinality of the set $\mathscr{P}_d^{(\delta)}(X)$ is asymptotically as large as the cardinality of the total set $\mathscr{P}_d(X)$.

\begin{lemma}
\label{Lemma restricting}
Let $d \geq 2$. Let also $\delta \in (0,1)$ and $\varepsilon > 0$. We have
\begin{equation*}
\# \mathscr{P}_d^{(\delta)}(X) = \# \mathscr{P}_d(X) \left( 1 + O \left( \frac1{X^{\delta-\varepsilon}} \right) \right),
\end{equation*}
where the implied constant depends at most on $d$, $\delta$ and $\varepsilon$.
\end{lemma}

\begin{proof}
We start by noting that
\begin{equation*}
\# \mathscr{P}_d^{(\delta)}(X) - \# \mathscr{P}_d(X) \ll X^{d-2} \sum_{0 < a_0, |a_d| \leq X} \# \left\{ \ell \mid a_d : \begin{array}{l l}
\rad(\ell) \mid a_0 \\
\ell > X^{\delta}
\end{array}
\right\}.
\end{equation*}
Therefore, we get
\begin{align*}
\# \mathscr{P}_d^{(\delta)}(X) - \# \mathscr{P}_d(X) & \ll X^{d-2} \sum_{X^{\delta} < \ell \leq X} \ \sum_{\substack{0 < a_0 \leq X \\ \rad(\ell) \mid a_0}} \ \sum_{\substack{0 < |a_d| \leq X \\ \ell \mid a_d}} 1 \\
& \ll X^d \sum_{X^{\delta} < \ell \leq X} \frac1{\ell \cdot \rad(\ell)} \\
& \ll X^{d-\delta} \sum_{\ell \leq X} \frac1{\rad(\ell)}.
\end{align*}
As a result, Lemma~\ref{Lemma dB} gives
\begin{equation*}
\# \mathscr{P}_d^{(\delta)}(X) - \# \mathscr{P}_d(X) \ll X^{d-\delta+ \varepsilon}.
\end{equation*}
Recalling the estimate \eqref{Estimate P}, we see that this completes the proof.
\end{proof}

\subsection{On the average number of rational numbers with small canonical height}

The following result will be the key tool in the proof of our results and we will invoke it repeatedly. More precisely, we will use it with $k=3$ to obtain the error term stated in Theorem~\ref{Theorem UBC}, and we will apply it with $k=2$ in the proof of Theorem~\ref{Theorem Lang} and Proposition~\ref{Proposition lambda statistical}.

\begin{lemma}
\label{Lemma useful}
Let $d \geq 3$ and $k \in \{2,3\}$. Let also $\delta \in (0,1)$ and $\varepsilon > 0$. Finally, let
\begin{equation*}
\alpha_{d,k} = \frac{(k-1)^2(d-1)}{k(d+k-2)}.
\end{equation*}
For $X, T \geq 1$, we have
\begin{equation*}
\sum_{\psi_{\mathbf{a}} \in \mathscr{P}_d^{(\delta)}(X)} \# \left\{ z \in \mathbb{Q}^{\times} : \hat{h}_{\psi_{\mathbf{a}}}(z) \leq \log T \right\} \ll X^{d-1 + 1/k + \delta \alpha_{d,k}/2 + \varepsilon} T^{2+d\alpha_{d,k}},
\end{equation*}
and
\begin{equation*}
\sum_{\psi_{\mathbf{a}} \in \mathscr{P}_d^{(\delta)}(X)} \# \left\{ z \in \mathbb{Q} : \hat{h}_{\psi_{\mathbf{a}}}(z) \leq \log T \right\} \ll X^{d-1+1/k +\delta \alpha_{d,k} /2 + \varepsilon} T^{d^2 \alpha_{d,k}},
\end{equation*}
where the implied constants depend at most on $d$, $\delta$ and $\varepsilon$.
\end{lemma}

\begin{proof}
We start by proving the first statement. Recall the definition \eqref{Definition r} of the quantity $r_{\infty}(\psi_{\mathbf{a}})$. We note that for $\psi_{\mathbf{a}} \in \mathscr{P}_d^{(\delta)}(X)$ we have $|a_i| \leq X$ for any $i \in \{0, \dots, d-2\}$, so the assumption $d \geq 3$ implies that
\begin{equation*}
r_{\infty}(\psi_{\mathbf{a}}) \leq \frac{X^{1/2}}{|a_d|^{1/2}}.
\end{equation*}
In addition, we recall that since $\psi_{\mathbf{a}} \in \mathscr{P}_d^{(\delta)}(X)$ we have in particular $\Delta_{\mathrm{s}_k(a_0)}(a_d) \leq X^{\delta}$. As a result, applying Lemmas~\ref{Lemma Weil} and \ref{Lemma a0} and letting respectively $x \in \mathbb{Z}$ and $y \geq 1$ denote the numerator and the denominator of $z \in \mathbb{Q}^{\times}$, we find that
\begin{equation*}
\# \mathscr{B}^{\times}_{\psi_{\mathbf{a}}}(T) \leq \# \left\{ z \in \mathbb{Q}^{\times} : 
\begin{array}{l}
\displaystyle{x \ll \frac{X^{1/2} T \Delta_y(a_d)^{1/2}}{|a_d|^{1/2}}, \ y \leq T \Delta_y(a_d)^{1/2}} \\
\vspace{-8pt} \\
\displaystyle{\gcd(\mathrm{s}_k(a_0), \sigma_{\mathbf{a}}(x,y)) \geq \frac{\mathrm{s}_k(a_0)}{A(X,T) \gcd(\mathrm{s}_k(a_0), a_d)^{\beta_{d,k}}}}
\end{array}
\right\},
\end{equation*}
where we have set
\begin{equation*}
\mathscr{B}^{\times}_{\psi_{\mathbf{a}}}(T) = \left\{ z \in \mathbb{Q}^{\times} : \hat{h}_{\psi_{\mathbf{a}}}(z) \leq \log T \right\},
\end{equation*}
and
\begin{equation}
\label{Definition A}
A(X,T) = \left( X^{\delta/2} T^d \right)^{k \alpha_{d,k}/(k-1)},
\end{equation}
and also
\begin{equation}
\label{Definition beta}
\beta_{d,k} = \frac{k-1}{d+k-2}.
\end{equation}
We thus have
\begin{align*}
\sum_{|a_1| \leq X} \# \mathscr{B}^{\times}_{\psi_{\mathbf{a}}}(T) \ll & \
\sum_{\substack{D \mid \mathrm{s}_k(a_0) \\ D \geq \mathrm{s}_k(a_0)/A(X,T) \gcd(\mathrm{s}_k(a_0), a_d)^{\beta_{d,k}}}} \sum_{0 < y \leq T \Delta_y(a_d)^{1/2}} \\
& \sum_{\substack{0 < |x| \ll X^{1/2} T \Delta_y(a_d)^{1/2} / |a_d|^{1/2} \\ \gcd(x,y) = 1}}
\sum_{\substack{|a_1| \leq X \\ \sigma_{\mathbf{a}}(x,y) \equiv 0 \bmod{D}}} 1.
\end{align*}
Since $D$ is cubefree and $\gcd(x,y) = 1$, the congruence $\sigma_{\mathbf{a}}(x,y) \equiv 0 \bmod{D}$ implies that
\begin{equation*}
\gcd(y^{d-1}, D) \mid \gcd(D,a_d).
\end{equation*}
Using the fact that $D \leq X$, we thus see that we have in particular
\begin{equation*}
\sum_{\substack{|a_1| \leq X \\ \sigma_{\mathbf{a}}(x,y) \equiv 0 \bmod{D}}} 1 \ll \frac{X \gcd(x,D) \gcd(D,a_d)}{D}.
\end{equation*}
Therefore, letting $\tau$ denote the divisor function, the summation over $x$ yields
\begin{align}
\begin{split}
\label{Upper bound sum B}
\sum_{|a_1| \leq X} \# \mathscr{B}^{\times}_{\psi_{\mathbf{a}}}(T) \ll & \ \frac{X^{3/2} T}{|a_d|^{1/2}}
\left( \sum_{0 < y \leq T \Delta_y(a_d)^{1/2}} \Delta_y(a_d)^{1/2} \right) \\
& \sum_{\substack{D \mid \mathrm{s}_k(a_0) \\ D \geq \mathrm{s}_k(a_0)/A(X,T) \gcd(\mathrm{s}_k(a_0), a_d)^{\beta_{d,k}}}} \frac{\gcd(D,a_d) \tau(D)}{D}.
\end{split}
\end{align}
It is convenient to introduce the arithmetic function $\xi$ defined for nonzero $n \in \mathbb{Z}$ by
\begin{equation}
\label{Definition xi}
\xi(n) = \sum_{\ell \mid n} \frac{\ell}{\rad(\ell)}.
\end{equation}
Recalling the definition \eqref{Definition Delta} of the largest divisor $\Delta_y(a_d)$ of $a_d$ whose radical divides $y$, we deduce
\begin{align*}
\sum_{0 < y \leq T \Delta_y(a_d)^{1/2}} \Delta_y(a_d)^{1/2} & \ll \sum_{\ell \mid a_d} \ell^{1/2} \sum_{\substack{0 < y \leq T \ell^{1/2} \\ \rad(\ell) \mid y}} 1 \\
& \ll T \xi(a_d).
\end{align*}
Moreover, we recall that the divisor bound states that for any $n \geq 1$ and $\varepsilon > 0$, we have
\begin{equation}
\label{Upper bound divisor}
\tau(n) \ll n^{\varepsilon}.
\end{equation}
Therefore, recalling the upper bound \eqref{Upper bound sum B} we see that
\begin{equation*}
\sum_{|a_1| \leq X} \# \mathscr{B}^{\times}_{\psi_{\mathbf{a}}}(T) \ll \frac{X^{3/2 + \varepsilon} T^2 \xi(a_d)}{|a_d|^{1/2}}
\sum_{\substack{D \mid \mathrm{s}_k(a_0) \\ D \geq \mathrm{s}_k(a_0)/A(X,T) \gcd(\mathrm{s}_k(a_0), a_d)^{\beta_{d,k}}}} \frac{\gcd(D,a_d)}{D}.
\end{equation*}
We now split the sum over $a_0$ into sums over its $k$-full part $a_0/\mathrm{s}_k(a_0)$ and its $k$-free part $\mathrm{s}_k(a_0)$. The classical upper bound for the number of $k$-full numbers with bounded absolute value asserts that
\begin{equation}
\label{Upper bound kfull}
\# \left\{ b \in \mathbb{Z} :
\begin{array}{l l}
0 < b \leq Y \\
\forall p \ \ p \mid b \implies p^k \mid b
\end{array}
\right\} \ll Y^{1/k}.
\end{equation}
As a result, we deduce that
\begin{equation*}
\sum_{a_0, |a_1| \leq X} \# \mathscr{B}^{\times}_{\psi_{\mathbf{a}}}(T) \ll \frac{X^{3/2 + 1/k + \varepsilon} T^2 \xi(a_d)}{|a_d|^{1/2}} \sum_{D \leq X} \frac{\gcd(D,a_d)}{D} \sum_{\substack{s \leq D A(X,T) \gcd(s, a_d)^{\beta_{d,k}} \\ D \mid s}} \frac1{s^{1/k}}.
\end{equation*}
But we easily get
\begin{align*}
\sum_{\substack{s \leq D A(X,T) \gcd (s, a_d)^{\beta_{d,k}} \\ D \mid s}} \frac1{s^{1/k}} & \ll \frac1{D^{1/k}} \sum_{u \leq A(X,T) \gcd(Du, a_d)^{\beta_{d,k}}} \frac1{u^{1/k}} \\
& \ll \frac1{D^{1/k}} \sum_{\ell \mid a_d} \frac1{\ell^{1/k}} \sum_{v \leq A(X,T) \gcd(D, a_d)^{\beta_{d,k}}/\ell^{1-\beta_{d,k}}} \frac1{v^{1/k}} \\ 
& \ll \frac{A(X,T)^{1-1/k} \gcd(D, a_d)^{(k-1)\beta_{d,k}/k}}{D^{1/k}} \sum_{m \mid a_d} \frac1{m^{1-(k-1)\beta_{d,k}/k}}.
\end{align*}
Moreover, we note that it follows from the definition \eqref{Definition beta} of the quantity $\beta_{d,k}$ and our assumptions $d \geq 3$ and $k \in \{2,3\}$ that
\begin{equation}
\label{Upper bound beta}
\beta_{d,k} \leq \frac1{k-1}.
\end{equation}
Hence, an application of the divisor bound \eqref{Upper bound divisor} gives
\begin{equation}
\label{Upper bound sum s}
\sum_{\substack{s \leq D A(X,T) \gcd (s, a_d)^{\beta_{d,k}} \\ D \mid s}} \frac1{s^{1/k}} \ll \frac{X^{\varepsilon} A(X,T)^{1-1/k} \gcd(D, a_d)^{(k-1)\beta_{d,k}/k}}{D^{1/k}}.
\end{equation}
Recalling the definition \eqref{Definition A} of the quantity $A(X,T)$, we see that we have obtained
\begin{equation*}
\sum_{a_0, |a_1| \leq X} \# \mathscr{B}^{\times}_{\psi_{\mathbf{a}}}(T) \ll \frac{X^{3/2 + 1/k + \delta \alpha_{d,k}/2 + 2 \varepsilon} T^{2+d\alpha_{d,k}} \xi(a_d)}{|a_d|^{1/2}} \sum_{D \leq X} \frac{\gcd(D,a_d)^{1+(k-1)\beta_{d,k}/k}}{D^{1+1/k}}.
\end{equation*}
But we have
\begin{equation*}
\sum_{D \leq X} \frac{\gcd(D,a_d)^{1+(k-1)\beta_{d,k}/k}}{D^{1+1/k}} \ll \sum_{m \mid a_d} \frac1{m^{1/k-(k-1)\beta_{d,k}/k}}.
\end{equation*}
Therefore, using the upper bound \eqref{Upper bound beta} and applying the divisor bound \eqref{Upper bound divisor} once again, we derive
\begin{equation*}
\sum_{\psi_{\mathbf{a}} \in \mathscr{P}_d^{(\delta)}(X)} \# \mathscr{B}^{\times}_{\psi_{\mathbf{a}}}(T) \ll X^{d-3/2 + 1/k + \delta \alpha_{d,k}/2 + 3 \varepsilon} T^{2+d\alpha_{d,k}} \sum_{0 < |a_d| \leq X} \frac{\xi(a_d)}{|a_d|^{1/2}}.
\end{equation*}
But recalling the definition \eqref{Definition xi} of the arithmetic function $\xi$, we observe that
\begin{align*}
\sum_{0 < |a_d| \leq X} \frac{\xi(a_d)}{|a_d|^{1/2}} & \ll \sum_{\ell \leq X} \frac{\ell^{1/2}}{\rad(\ell)} \sum_{0 < |a| \leq X/\ell} \frac1{|a|^{1/2}} \\
& \ll X^{1/2} \sum_{\ell \leq X} \frac1{\rad(\ell)}.
\end{align*}
As a result, we see that an application of Lemma~\ref{Lemma dB} completes the proof of the first statement.

Next, we establish the second statement. We start by noting that the assumptions $d \geq 3$ and $k \in \{2,3\}$ imply that
\begin{equation*}
2 + d \alpha_{d,k} \leq d^2 \alpha_{d,k}.
\end{equation*}
Hence, we see that it suffices to prove that
\begin{equation}
\label{Upper bound z=0}
\sum_{\substack{\psi_{\mathbf{a}} \in \mathscr{P}_d^{(\delta)}(X) \\ \hat{h}_{\psi_{\mathbf{a}}}(0) \leq \log T }} 1 \ll X^{d-1+1/k +\delta \alpha_{d,k} /2 + \varepsilon} T^{d^2 \alpha_{d,k}}.
\end{equation}
The equalities $\psi_{\mathbf{a}}(0)=1$ and \eqref{Equality global height} show that if $\hat{h}_{\psi_{\mathbf{a}}}(0) \leq \log T$ then $\hat{h}_{\psi_{\mathbf{a}}}(1) \leq \log T^d$. In addition, since $\psi_{\mathbf{a}} \in \mathscr{P}_d^{(\delta)}(X)$ we have in particular $\Delta_{\mathrm{s}_k(a_0)}(a_d) \leq X^{\delta}$. An application of Lemma~\ref{Lemma a0} with $z=1$ thus yields
\begin{equation*}
\sum_{\substack{\psi_{\mathbf{a}} \in \mathscr{P}_d^{(\delta)}(X) \\ \hat{h}_{\psi_{\mathbf{a}}}(0) \leq \log T }} 1 \leq
\sum_{\psi_{\mathbf{a}} \in \mathscr{P}_d^{(\delta)}(X)}
\# \left\{ D \mid \mathrm{s}_k(a_0) : 
\begin{array}{l}
\displaystyle{D \geq \frac{\mathrm{s}_k(a_0)}{B(X,T) \gcd (\mathrm{s}_k(a_0), a_d)^{\beta_{d,k}}}} \\
\vspace{-8pt} \\
\sigma_{\mathbf{a}}(1,1) \equiv 0 \bmod{D}
\end{array}
\right\},
\end{equation*}
where we have set
\begin{equation}
\label{Definition B}
B(X,T) = \left( X^{\delta/2} T^{d^2} \right)^{k \alpha_{d,k}/(k-1)}.
\end{equation}
We proceed to carry out the summation over $a_1$ first. Using the fact that $D \leq X$, we get
\begin{align*}
\sum_{\substack{\psi_{\mathbf{a}} \in \mathscr{P}_d^{(\delta)}(X) \\ \hat{h}_{\psi_{\mathbf{a}}}(0) \leq \log T }} 1 & \leq
\sum_{\substack{|a_2|, \dots, |a_{d-2}| \leq X \\ 0 < a_0, |a_d| \leq X}} \
\sum_{\substack{D \mid \mathrm{s}_k(a_0) \\ D \geq \mathrm{s}_k(a_0) / B(X,T) \gcd (\mathrm{s}_k(a_0), a_d)^{\beta_{d,k}}}} \
\sum_{\substack{|a_1| \leq X \\ \sigma_{\mathbf{a}}(1,1) \equiv 0 \bmod{D}}} 1 \\
& \ll X \sum_{\substack{|a_2|, \dots, |a_{d-2}| \leq X \\ 0 < a_0, |a_d| \leq X}} \
\sum_{\substack{D \mid \mathrm{s}_k(a_0) \\ D \geq \mathrm{s}_k(a_0) / B(X,T) \gcd (\mathrm{s}_k(a_0), a_d)^{\beta_{d,k}}}} \frac1{D}.
\end{align*}
Once again we handle the summation over $a_0$ by first summing over its $k$-full part $a_0/\mathrm{s}_k(a_0)$ using the upper bound \eqref{Upper bound kfull}. This yields
\begin{equation*}
\sum_{\substack{\psi_{\mathbf{a}} \in \mathscr{P}_d^{(\delta)}(X) \\ \hat{h}_{\psi_{\mathbf{a}}}(0) \leq \log T}} 1 \leq X^{d-2+1/k} \sum_{0 < |a_d| \leq X} \ \sum_{D \leq X} \frac1{D} \sum_{\substack{s \leq D B(X,T) \gcd (s,a_d)^{\beta_{d,k}} \\ D \mid s}} \frac1{s^{1/k}}.
\end{equation*}
Recalling the definition \eqref{Definition B} of the quantity $B(X,T)$ and appealing to the upper bound \eqref{Upper bound sum s} with $A(X,T)$ replaced by $B(X,T)$, we derive
\begin{equation*}
\sum_{\substack{\psi_{\mathbf{a}} \in \mathscr{P}_d^{(\delta)}(X) \\ \hat{h}_{\psi_{\mathbf{a}}}(0) \leq \log T }} 1 \leq X^{d-2+1/k +\delta \alpha_{d,k} /2 + \varepsilon} T^{d^2 \alpha_{d,k}} \sum_{0 < |a_d| \leq X} \sum_{D \leq X} \frac{ \gcd(D, a_d)^{(k-1)\beta_{d,k}/k}}{D^{1+1/k}}.
\end{equation*}
But using the upper bound \eqref{Upper bound beta} we easily get
\begin{equation*}
\sum_{0 < |a_d| \leq X} \sum_{D \leq X} \frac{ \gcd(D, a_d)^{(k-1)\beta_{d,k}/k}}{D^{1+1/k}} \ll X.
\end{equation*}
We thus see that the upper bound \eqref{Upper bound z=0} follows, which completes the proof.
\end{proof}

\section{The uniform boundedness conjecture on average}

\label{Section UBC}

Our goal in this section is to furnish the proof of Theorem~\ref{Theorem UBC}. The following result is a direct consequence of the work of Benedetto \cite{MR2339471} and provides us with a convenient bound for the number of rational preperiodic points of a polynomial $\psi_{\mathbf{a}} \in \mathscr{P}_d$.

\begin{lemma}
\label{Lemma Benedetto}
Let $d \geq 2$ and $\psi_{\mathbf{a}} \in \mathscr{P}_d$. We have
\begin{equation*}
\# \mathrm{Prep}_{\mathbb{Q}}(\psi_{\mathbf{a}}) \ll 1 + \log \mathscr{H}(\psi_{\mathbf{a}}),
\end{equation*}
where the implied constant depends at most on $d$.
\end{lemma}

\begin{proof}
The set of prime numbers where the polynomial $\psi_{\mathbf{a}}$ does not have potentially good reduction (see \cite[Definition~2.1]{MR2339471}) is contained in the set of prime divisors of $a_0 a_d$. Therefore, letting $\omega(n)$ denote the number of prime numbers dividing an integer $n \geq 1$ and appealing to Benedetto's result \cite[Main Theorem]{MR2339471}, we get
\begin{equation*}
\# \mathrm{Prep}_{\mathbb{Q}}(\psi_{\mathbf{a}}) \ll 1 + \omega(a_0 |a_d|) \log \left( 1 + \omega(a_0 |a_d|) \right).
\end{equation*}
Invoking the classical upper bound
\begin{equation*}
\omega(n) \ll \frac{\log n}{\log \log n},
\end{equation*}
we deduce
\begin{equation*}
\# \mathrm{Prep}_{\mathbb{Q}}(\psi_{\mathbf{a}}) \ll 1 + \log (a_0 |a_d|).
\end{equation*}
Since $a_0, |a_d| \leq \mathscr{H}(\psi_{\mathbf{a}})$, we see that this completes the proof.
\end{proof}

We are now ready to reveal the proof of Theorem~\ref{Theorem UBC}.

\begin{proof}[Proof of Theorem~\ref{Theorem UBC}]
We start by handling the case $d=2$. It follows from Lemma~\ref{Lemma d=2} that if a polynomial $\psi_{\mathbf{a}} \in \mathscr{P}_2$ has a rational preperiodic point then the integer $a_0$ must be squareful. As a result, combining Lemma~\ref{Lemma Benedetto} with the upper bound \eqref{Upper bound kfull}, we get
\begin{equation*}
\sum_{\psi_{\mathbf{a}} \in \mathscr{P}_2(X)} \# \mathrm{Prep}_{\mathbb{Q}}(\psi_{\mathbf{a}}) \ll X^{3/2} \log X.
\end{equation*}
Recalling the estimate \eqref{Estimate P} we thus derive
\begin{equation*}
\frac1{\# \mathscr{P}_2(X)} \sum_{\psi_{\mathbf{a}} \in \mathscr{P}_2(X)} \# \mathrm{Prep}_{\mathbb{Q}}(\psi_{\mathbf{a}}) \ll \frac{\log X}{X^{1/2}},
\end{equation*}
which completes the proof of Theorem~\ref{Theorem UBC} in the case $d=2$.

We now deal with the case $d \geq 3$. Given $\delta \in (0,1)$, recall the definition \eqref{Definition Pddelta} of the set $\mathscr{P}_d^{(\delta)}(X)$. Combining Lemmas~\ref{Lemma restricting} and \ref{Lemma Benedetto} we deduce that for any $\delta \in (0,1)$, we have
\begin{equation*}
\sum_{\psi_{\mathbf{a}} \in \mathscr{P}_d(X)} \# \mathrm{Prep}_{\mathbb{Q}}(\psi_{\mathbf{a}}) =
\sum_{\psi_{\mathbf{a}} \in \mathscr{P}_d^{(\delta)}(X)} \# \mathrm{Prep}_{\mathbb{Q}}(\psi_{\mathbf{a}}) + O \left( X^{d -\delta + \varepsilon} \right).
\end{equation*}
Applying Lemma~\ref{Lemma useful} with $T=1$, we see that for $k \in \{2,3 \}$, we have
\begin{equation*}
\sum_{\psi_{\mathbf{a}} \in \mathscr{P}_d(X)} \# \mathrm{Prep}_{\mathbb{Q}}(\psi_{\mathbf{a}}) \ll X^{d + \varepsilon} \left( \frac1{X^{1-1/k - \delta \alpha_{d,k}/2}} + \frac1{X^{\delta}} \right).
\end{equation*}
An easy calculation shows that the optimal choice of $\delta \in (0,1)$ is given by
\begin{equation*}
\delta = \frac{2(k-1)(d+k-2)}{k^2(d+1)-2k+d-1}.
\end{equation*}
Choosing $k = 3$ and using the estimate \eqref{Estimate P}, we see that this completes the proof of Theorem~\ref{Theorem UBC} in the case $d \geq 3$.
\end{proof}

\section{Dynamical successive minima of random polynomials}

\subsection{A statistical version of the dynamical Lang conjecture}

\label{Section Lang}

Our aim in this section is to establish Theorem~\ref{Theorem Lang}. The following result provides us with a sharp upper bound for the quantity $\lambda_1(\psi_{\mathbf{a}})$ under a mild assumption on the polynomial $\psi_{\mathbf{a}} \in \mathscr{P}_d$.

\begin{lemma}
\label{Lemma lambda1}
Let $d \geq 2$ and $\psi_{\mathbf{a}} \in \mathscr{P}_d$. If $0 \notin \mathrm{Prep}_{\mathbb{Q}}(\psi_{\mathbf{a}})$ then
\begin{equation*}
\lambda_1(\psi_{\mathbf{a}}) \leq \frac1{d(d-1)} \log \mathscr{H}(\psi_{\mathbf{a}}) + \frac{\log d}{d(d-1)}.
\end{equation*}
\end{lemma}

\begin{proof}
An application of Lemma~\ref{Lemma upper bound} gives
\begin{equation*}
\hat{h}_{\psi_{\mathbf{a}}}(1) \leq \frac1{d-1} \log \mathscr{H}(\psi_{\mathbf{a}}) + \frac{\log d}{d-1}.
\end{equation*}
Therefore, the equalities $\psi_{\mathbf{a}}(0) = 1$ and \eqref{Equality global height} yield
\begin{equation*}
\hat{h}_{\psi_{\mathbf{a}}}(0) \leq \frac1{d(d-1)} \log \mathscr{H}(\psi_{\mathbf{a}}) + \frac{\log d}{d(d-1)},
\end{equation*}
which completes the proof since by assumption $0 \notin \mathrm{Prep}_{\mathbb{Q}}(\psi_{\mathbf{a}})$.
\end{proof}

We now have all the tools required to prove Theorem~\ref{Theorem Lang}.

\begin{proof}[Proof of Theorem~\ref{Theorem Lang}]
Appealing to Lemma~\ref{Lemma lambda1} we get
\begin{equation*}
\# \left\{ \psi_{\mathbf{a}} \in \mathscr{P}_d(X) : \lambda_1(\psi_{\mathbf{a}}) > \left( \frac1{d(d-1)} + \varepsilon \right) \log \mathscr{H}(\psi_{\mathbf{a}}) \right\} \leq \sum_{\substack{\psi_{\mathbf{a}} \in \mathscr{P}_d(X) \\ 0 \in \mathrm{Prep}_{\mathbb{Q}}(\psi_{\mathbf{a}})}} 1 + O(1).
\end{equation*}
It follows in particular that
\begin{equation*}
\# \left\{ \psi_{\mathbf{a}} \in \mathscr{P}_d(X) : \lambda_1(\psi_{\mathbf{a}}) > \left( \frac1{d(d-1)} + \varepsilon \right) \log \mathscr{H}(\psi_{\mathbf{a}}) \right\} \ll \sum_{\psi_{\mathbf{a}} \in \mathscr{P}_d(X)} \# \mathrm{Prep}_{\mathbb{Q}}(\psi_{\mathbf{a}}).
\end{equation*}
Therefore, an application of Theorem~\ref{Theorem UBC} yields
\begin{equation*}
\lim_{X \to \infty} \frac1{\# \mathscr{P}_d(X)} \cdot \# \left\{ \psi_{\mathbf{a}} \in \mathscr{P}_d(X) : \lambda_1(\psi_{\mathbf{a}}) > \left( \frac1{d(d-1)} + \varepsilon \right) \log \mathscr{H}(\psi_{\mathbf{a}}) \right\} = 0.
\end{equation*}
We thus conclude that in order to complete the proof of Theorem~\ref{Theorem Lang}, it suffices to prove that for any $\varepsilon \in (0,1/d^2)$, we have
\begin{equation}
\label{Goal Lang}
\lim_{X \to \infty} \frac1{\# \mathscr{P}_d(X)} \cdot
\# \left\{ \psi_{\mathbf{a}} \in \mathscr{P}_d(X) : \lambda_1(\psi_{\mathbf{a}}) < \left( \frac1{d(d-1)} - \varepsilon \right) \log \mathscr{H}(\psi_{\mathbf{a}}) \right\} = 0.
\end{equation}

We start by handling the case $d =2$. Invoking Lemma~\ref{Lemma d=2} we deduce that
\begin{equation*}
\# \left\{ \psi_{\mathbf{a}} \in \mathscr{P}_2(X) : \lambda_1(\psi_{\mathbf{a}}) < \left( \frac1{2} - \varepsilon \right) \log \mathscr{H}(\psi_{\mathbf{a}}) \right\} \leq \sum_{\substack{\psi_{\mathbf{a}} \in \mathscr{P}_2(X) \\ \mathrm{s}_2(a_0) < X^{1-2\varepsilon}}} 1.
\end{equation*}
Summing first over the squareful part $a_0/\mathrm{s}_2(a_0)$ of the integer $a_0$ using the upper bound \eqref{Upper bound kfull}, we get
\begin{equation*}
\# \left\{ \psi_{\mathbf{a}} \in \mathscr{P}_2(X) : \lambda_1(\psi_{\mathbf{a}}) < \left( \frac1{2} - \varepsilon \right) \log \mathscr{H}(\psi_{\mathbf{a}}) \right\} \ll X^{3/2} \sum_{s < X^{1-2\varepsilon}} \frac1{s^{1/2}}.
\end{equation*}
Appealing to the estimate \eqref{Estimate P} we thus derive
\begin{equation*}
\frac1{\# \mathscr{P}_2(X)} \cdot \# \left\{ \psi_{\mathbf{a}} \in \mathscr{P}_2(X) : \lambda_1(\psi_{\mathbf{a}}) < \left( \frac1{2} - \varepsilon \right) \log \mathscr{H}(\psi_{\mathbf{a}}) \right\} \ll \frac1{X^{\varepsilon}}.
\end{equation*}
The equality \eqref{Goal Lang} follows, which completes the proof of Theorem~\ref{Theorem Lang} in the case $d=2$.

We now deal with the case $d \geq 3$. An application of Lemma~\ref{Lemma restricting} shows that for any $\delta \in (0,1)$, we have
\begin{align}
\begin{split}
\label{Estimate restricting}
\# \left\{ \psi_{\mathbf{a}} \in \mathscr{P}_d(X) : \lambda_1(\psi_{\mathbf{a}}) < \left( \frac1{d(d-1)} - \varepsilon \right) \log \mathscr{H}(\psi_{\mathbf{a}}) \right\} = & \ \# \mathscr{C}_{d,\varepsilon}^{(\delta)}(X) \\
& + O \left( X^{d-\delta + \varepsilon} \right),
\end{split}
\end{align}
where we have set
\begin{equation*}
\mathscr{C}_{d,\varepsilon}^{(\delta)}(X) = \left\{ \psi_{\mathbf{a}} \in \mathscr{P}_d^{(\delta)}(X) : \lambda_1(\psi_{\mathbf{a}}) < \left( \frac1{d(d-1)} - \varepsilon \right) \log \mathscr{H}(\psi_{\mathbf{a}}) \right\}.
\end{equation*}
Moreover, it is clear that for any $\psi_{\mathbf{a}} \in \mathscr{C}_{d,\varepsilon}^{(\delta)}(X)$, we have
\begin{equation*}
\# \left\{ z \in \mathbb{Q} : \hat{h}_{\psi_{\mathbf{a}}}(z) < \left( \frac1{d(d-1)} - \varepsilon \right) \log \mathscr{H}(\psi_{\mathbf{a}}) \right\} \geq 1.
\end{equation*}
It follows that
\begin{equation*}
\# \mathscr{C}_{d,\varepsilon}^{(\delta)}(X) \leq \sum_{\psi_{\mathbf{a}} \in \mathscr{P}_d^{(\delta)}(X)} \# \left\{ z \in \mathbb{Q} : \hat{h}_{\psi_{\mathbf{a}}}(z) < \left( \frac1{d(d-1)} - \varepsilon \right) \log \mathscr{H}(\psi_{\mathbf{a}}) \right\},
\end{equation*}
which implies in particular
\begin{equation*}
\# \mathscr{C}_{d,\varepsilon}^{(\delta)}(X) \leq \sum_{\psi_{\mathbf{a}} \in \mathscr{P}_d^{(\delta)}(X)} \# \left\{ z \in \mathbb{Q} : \hat{h}_{\psi_{\mathbf{a}}}(z) < \log X^{1/d(d-1) - \varepsilon} \right\}.
\end{equation*}
Appealing to Lemma~\ref{Lemma useful} with $k=2$, we thus get
\begin{equation*}
\# \mathscr{C}_{d,\varepsilon}^{(\delta)}(X) \ll X^{d - \varepsilon (-1+d(d-1)/2) + \delta (d-1)/4d}.
\end{equation*}
Since $\varepsilon \in (0, 1/d^2)$ we can choose $\delta = \varepsilon d^2$. Therefore, recalling the estimate \eqref{Estimate restricting}, we obtain
\begin{equation*}
\# \left\{ \psi_{\mathbf{a}} \in \mathscr{P}_d(X) : \lambda_1(\psi_{\mathbf{a}}) < \left( \frac1{d(d-1)} - \varepsilon \right) \log \mathscr{H}(\psi_{\mathbf{a}}) \right\} \ll X^{d - \varepsilon (-1+d(d-1)/4)}.
\end{equation*}
As a result, using the estimate \eqref{Estimate P} we eventually derive
\begin{equation*}
\frac1{\# \mathscr{P}_d(X)} \cdot \# \left\{ \psi_{\mathbf{a}} \in \mathscr{P}_d(X) : \lambda_1(\psi_{\mathbf{a}}) < \left( \frac1{d(d-1)} - \varepsilon \right) \log \mathscr{H}(\psi_{\mathbf{a}}) \right\} \ll \frac1{X^{\varepsilon(-1+d(d-1)/4)}}. 
\end{equation*}
The equality \eqref{Goal Lang} thus follows, which completes the proof of Theorem~\ref{Theorem Lang} in the case $d \geq 3$.
\end{proof}

\subsection{Dynamically independent vectors of rational numbers}

In this section we say that two polynomials $f, g \in \mathbb{Q}[z]$ are conjugate if there exists $(\alpha, \beta) \in \overline{\mathbb{Q}}^{\times} \times \overline{\mathbb{Q}}$ such that
\begin{equation*}
g(z) = \frac{f(\alpha z + \beta)}{\alpha} - \frac{\beta}{\alpha}.
\end{equation*}
In addition, given $\ell \geq 1$ we follow Medvedev and Scanlon \cite[Definition~2.23]{MR3126567} and we define the $\ell$-th Chebyshev polynomial $C_{\ell} \in \mathbb{Z}[z]$ as the unique polynomial satisfying the functional equation
\begin{equation}
\label{Definition C}
C_{\ell} \left( z + \frac1{z} \right) = z^{\ell} + \frac1{z^{\ell}}.
\end{equation}
Furthermore, we call a polynomial of the form $-C_{\ell}$ for some $\ell \geq 1$ a negative Chebyshev polynomial. Our first task is to establish the following result.

\begin{lemma}
\label{Lemma conjugate}
Let $d \geq 2$ and $\psi_{\mathbf{a}} \in \mathscr{P}_d$. If $0 \notin \mathrm{Prep}_{\mathbb{Q}}(\psi_{\mathbf{a}})$ then no iterate of $\psi_{\mathbf{a}}$ is conjugate to a monomial, a Chebyshev polynomial or a negative Chebyshev polynomial.
\end{lemma}

\begin{proof}
We start by recalling that for any $m \geq 1$ the polynomial $\psi_{\mathbf{a}}^m$ has degree $d^m$. We reason by contradiction and we thus assume that for some $n \geq 1$ and for some polynomial $\varphi(z) \in \{ z^{d^n}, C_{d^n}(z), - C_{d^n}(z) \}$, there exists $(\alpha, \beta) \in \overline{\mathbb{Q}}^{\times} \times \overline{\mathbb{Q}}$ such that
\begin{equation*}
\psi_{\mathbf{a}}^n(z) = \frac{\varphi( \alpha z + \beta)}{\alpha} - \frac{\beta}{\alpha}.
\end{equation*}
But it is immediate to check by induction that the coefficient of degree $d^n-1$ of the polynomial $\psi_{\mathbf{a}}^n$ is equal to $0$. Since $\varphi$ also has this property, we deduce that we must have $\beta=0$ and therefore
\begin{equation}
\label{Equality phi}
\psi_{\mathbf{a}}^n(z) = \frac{\varphi(\alpha z)}{\alpha}.
\end{equation}

If either $d$ is odd or $\varphi(z) = z^{d^n}$ then $\varphi(0) = 0$ and thus also $\psi_{\mathbf{a}}^n(0) = 0$, which contradicts our assumption that $0 \notin \mathrm{Prep}_{\mathbb{Q}}(\psi_{\mathbf{a}})$. Next, if $d$ is even then the polynomial $C_{d^n}$ is even and is therefore conjugate to $-C_{d^n}$. As a result, we can now assume that $d$ is even and $\varphi(z) = C_{d^n}(z)$. Observing that it follows from the definition \eqref{Definition C} of the polynomial $C_{d^n}$ that $C_{d^n}(0) = 2 (-1)^{d^n/2}$, we see that the equality \eqref{Equality phi} shows that we must have
\begin{equation}
\label{Equality even}
\psi_{\mathbf{a}}^n(z) = (-1)^{d^n/2} \frac{\psi_{\mathbf{a}}^n(0)}{2} C_{d^n} \left( \frac{2}{\psi_{\mathbf{a}}^n(0)} z \right).
\end{equation}
In addition, the definition \eqref{Definition C} of the polynomial $C_{d^n}$ also gives $C_{d^n}(2)=2$, which yields
\begin{equation}
\label{Equality 0 prep}
\psi_{\mathbf{a}}^{2n}(0) = (-1)^{d^n/2} \psi_{\mathbf{a}}^n(0).
\end{equation}
But the equality \eqref{Equality even} shows that the polynomial $\psi_{\mathbf{a}}^n$ is even, so composing the equality \eqref{Equality 0 prep} by $\psi_{\mathbf{a}}^n$ we eventually obtain $\psi_{\mathbf{a}}^{3n}(0) = \psi_{\mathbf{a}}^{2n}(0)$. Once again this contradicts our assumption that $0 \notin \mathrm{Prep}_{\mathbb{Q}}(\psi_{\mathbf{a}})$, which completes the proof.
\end{proof}

Given a polynomial $\psi_{\mathbf{a}} \in \mathscr{P}_d$ satisfying a mild assumption, the following result allows us to exhibit many $\psi_{\mathbf{a}}$-dynamically independent vectors of rational numbers by appealing to the work of Medvedev and Scanlon \cite{MR3126567}. This will play a key role in the proof of the upper bound in Proposition~\ref{Proposition lambda}.

\begin{lemma}
\label{Lemma independent}
Let $d \geq 2$ and $\psi_{\mathbf{a}} \in \mathscr{P}_d$. Let also $N \geq 1$ and $q_1, \dots, q_N$ be distinct prime numbers coprime to $a_0 a_d$. If $0 \notin \mathrm{Prep}_{\mathbb{Q}}(\psi_{\mathbf{a}})$ then the vector
\begin{equation*}
\left( \frac1{q_1}, \dots, \frac1{q_N} \right)
\end{equation*}
is $\psi_{\mathbf{a}}$-dynamically independent.
\end{lemma}

\begin{proof}
We reason by contradiction and we thus assume that there exists a nonzero polynomial $\mathbf{P} \in \overline{\mathbb{Q}}[X_1, \dots, X_N]$ satisfying the property \eqref{Property independence} and
\begin{equation*}
\mathbf{P} \left( \frac1{q_1}, \dots, \frac1{q_N} \right)= 0.
\end{equation*}
We introduce the algebraic set
\begin{equation*}
W = \left\{ (\xi_1, \dots, \xi_N) \in \mathbb{A}^N(\overline{\mathbb{Q}}) : \mathbf{P}(\xi_1, \dots, \xi_N) = 0 \right\},
\end{equation*}
and we let $V_0$ be an irreducible component of $W$ such that
\begin{equation}
\label{Vector V0}
\left( \frac1{q_1}, \dots, \frac1{q_N} \right) \in V_0.
\end{equation}
We also let
\begin{equation*}
\Psi_{\mathbf{a}} = (\psi_{\mathbf{a}}, \dots, \psi_{\mathbf{a}}) \in \mathbb{Q}[z]^N.
\end{equation*}
The fact that the polynomial $\mathbf{P}$ satisfies the property \eqref{Property independence} is equivalent to saying that
\begin{equation*}
\Psi_{\mathbf{a}}(W) \subseteq W.
\end{equation*}
Therefore, since the set $\Psi_{\mathbf{a}}(V_0)$ is irreducible there exists an irreducible component $V_1$ of $W$ such that
\begin{equation*}
\Psi_{\mathbf{a}}(V_0) \subseteq V_1.
\end{equation*}
Iterating this process we recursively define a sequence $(V_j)_{j \geq 0}$ of irreducible components of $W$. Next, we select $n \geq 0$ and $m > n$ such that $V_m = V_n$. By construction, the algebraic set $V_n$ is irreducible and satisfies
\begin{equation*}
\Psi_{\mathbf{a}}^{m-n}(V_n) \subseteq V_n.
\end{equation*}
In addition, by assumption we have $0 \notin \mathrm{Prep}_{\mathbb{Q}}(\psi_{\mathbf{a}})$ so it follows from Lemma~\ref{Lemma conjugate} that the polynomial $\psi_{\mathbf{a}}^{m-n}$ is not conjugate to a monomial, a Chebyshev polynomial or a negative Chebyshev polynomial. As a result, we are in position to appeal to the work of Medvedev and Scanlon \cite[Theorem]{MR3126567}. We deduce that there exist $R \geq 1$ and $(i_1, j_1), \dots, (i_R, j_R) \in \{1, \dots, N\}^2$, and also $g_1, \dots, g_R \in \overline{\mathbb{Q}}[z]$ commuting with $\psi_{\mathbf{a}}^{m-n}$ such that
\begin{equation}
\label{Equality MS}
V_n = \left\{ (\xi_1, \dots, \xi_N) \in \mathbb{A}^N(\overline{\mathbb{Q}}) : \forall r \in \{1, \dots, R\} \ \ \xi_{i_r} = g_r(\xi_{j_r}) \right\}.
\end{equation}

Let $r \in  \{1, \dots, R\}$. Recalling that $V_0$ satisfies the assumption \eqref{Vector V0}, we see that by construction we have
\begin{equation*}
\Psi_{\mathbf{a}}^n \left( \frac1{q_1}, \dots, \frac1{q_N} \right) \in V_n.
\end{equation*}
Hence, since the polynomials $g_r$ and $\psi_{\mathbf{a}}^{m-n}$ commute, for any $L \geq 0$ we have
\begin{equation}
\label{Equality iterates}
\psi_{\mathbf{a}}^{L(m-n)+n} \left( \frac1{q_{i_r}} \right) = g_r \left( \psi_{\mathbf{a}}^{L(m-n)+n} \left( \frac1{q_{j_r}} \right) \right).
\end{equation}
But by assumption $q_{i_r}$ is coprime to $a_0 a_d$ so it is straightforward to check that for any $s \geq 0$, we have
\begin{equation}
\label{Equality norm}
\left| \psi_{\mathbf{a}}^s \left( \frac1{q_{i_r}} \right) \right|_{q_{i_r}} = q_{i_r}^{d^s}.
\end{equation}
It follows that for any $L \geq 0$, we have
\begin{equation*}
\left| g_r \left( \psi_{\mathbf{a}}^{L(m-n)+n} \left( \frac1{q_{j_r}} \right) \right) \right|_{q_{i_r}} = q_{i_r}^{d^{L(m-n)+n}}.
\end{equation*}
Since $m>n$ this implies in particular that
\begin{equation*}
\lim_{L \to \infty} \left| \psi_{\mathbf{a}}^{L(m-n)+n} \left( \frac1{q_{j_r}} \right) \right|_{q_{i_r}} = \infty.
\end{equation*}
Using again the assumption that $q_{i_r}$ is coprime to $a_0$ we deduce that $q_{i_r} = q_{j_r}$, and therefore $i_r=j_r$ since the prime numbers $q_1, \dots, q_N$ are distinct. The equality \eqref{Equality iterates} thus shows that
\begin{equation*}
\left\{ \psi_{\mathbf{a}}^{L(m-n)+n} \left( \frac1{q_{i_r}} \right) : L \geq 0 \right\} \subseteq \left\{ z \in \mathbb{Q} : z-g_r(z) = 0 \right\}.
\end{equation*}
But the fact that $m>n$ and the equality \eqref{Equality norm} imply that the left-hand side is an infinite set, so we finally obtain $g_r(z) = z$.

We have thus proved that for any $r \in \{1, \dots, R\}$, we have $i_r=j_r$ and $g_r(z)=z$. Recalling the equality \eqref{Equality MS} we see that it eventually follows that $V_n = \mathbb{A}^N(\overline{\mathbb{Q}})$ and therefore $W = \mathbb{A}^N(\overline{\mathbb{Q}})$. This contradicts the fact that the polynomial $\mathbf{P}$ is nonzero, which completes the proof.
\end{proof}

\subsection{Dynamical successive minima of higher order}

\label{Section higher}

Our purpose in this section is to establish Propositions~\ref{Proposition lambda} and \ref{Proposition lambda statistical} and Theorem~\ref{Theorem N fixed}. Given $N \geq 1$ and $\psi_{\mathbf{a}} \in \mathscr{P}_d$, recall that the $N$-th dynamical successive minimum $\lambda_N(\psi_{\mathbf{a}})$ of the polynomial $\psi_{\mathbf{a}}$ was introduced in Definition~\ref{Definition successive}. Our first task is to complement Lemma~\ref{Lemma lambda1} by providing upper and lower bounds for the quantity $\lambda_N(\psi_{\mathbf{a}})$ for $N \geq 2$ under a mild assumption on the polynomial $\psi_{\mathbf{a}} \in \mathscr{P}_d$.

\begin{proposition}
\label{Proposition lambda}
Let $d \geq 3$ and $\psi_{\mathbf{a}} \in \mathscr{P}_d$. Let also $N \geq 2$ and $\varepsilon > 0$. If $0 \notin \mathrm{Prep}_{\mathbb{Q}}(\psi_{\mathbf{a}})$ then
\begin{equation*}
\frac1{2} \log \frac{N}{\mathscr{H}(\psi_{\mathbf{a}})} + O(1) < \lambda_N(\psi_{\mathbf{a}}) \leq \left( 1 + \varepsilon \right) \log N + \left( \frac1{d-1} + \varepsilon \right) \log \mathscr{H}(\psi_{\mathbf{a}}) + O(1),
\end{equation*}
where the implied constants depend at most on $d$ and $\varepsilon$.
\end{proposition}

\begin{proof}
We start by proving the lower bound. Let $(z_1, \dots, z_N) \in \mathbb{Q}^N$ be a $\psi_{\mathbf{a}}$-dynamically independent vector. It is clear that for any $i, j \in \{1, \dots, N\}$ such that $i \neq j$ we have $z_i \neq z_j$. As a result, since
\begin{equation*}
\# \left\{ z \in \mathbb{Q} : h(z) \leq \frac1{2} \log N - \log 2 \right\} < N,
\end{equation*}
it follows that there exists $i_0 \in \{1, \dots, N \}$ such that
\begin{equation*}
h(z_{i_0}) > \frac1{2} \log N - \log 2.
\end{equation*}
Therefore, Lemma~\ref{Lemma lower bound} yields
\begin{equation*}
\hat{h}_{\psi_{\mathbf{a}}}(z_{i_0}) > \frac1{2} \log N - \frac1{2} \log \mathscr{H}(\psi_{\mathbf{a}}) - \log d - \log 2,
\end{equation*}
which completes the proof of the claimed lower bound.

We now turn to the proof of the upper bound. We start by noting that Lemma~\ref{Lemma upper bound} shows that for any prime number $p$, we have
\begin{equation}
\label{Upper bound canonical}
\hat{h}_{\psi_{\mathbf{a}}} \left( \frac1{p} \right) \leq \frac1{d-1} \log \mathscr{H}(\psi_{\mathbf{a}}) + \log p + \frac{\log d}{d-1}.
\end{equation}
By assumption we have $0 \notin \mathrm{Prep}_{\mathbb{Q}}(\psi_{\mathbf{a}})$ so we are in position to appeal to Lemma~\ref{Lemma independent}. Recall that $\omega(n)$ denotes the number of prime numbers dividing an integer $n \geq 1$. In addition, for any $m \geq 1$ we let $p_m$ be the $m$-th prime number. Combining Lemma~\ref{Lemma independent} and the inequality \eqref{Upper bound canonical} we deduce that
\begin{equation}
\label{Upper bound lambda}
\lambda_N(\psi_{\mathbf{a}}) \leq \frac1{d-1} \log \mathscr{H}(\psi_{\mathbf{a}}) + \log p_{N + \omega(a_0 |a_d|)} + \frac{\log d}{d-1}.
\end{equation}
Furthermore, using Chebyshev's result stating that for any $m \geq 1$ we have $p_{m+1} \leq 2 p_m$, we derive
\begin{equation*}
p_{N + \omega(a_0 |a_d|)} \leq 2^{\omega(a_0 |a_d|)} p_N.
\end{equation*}
Hence, the divisor bound \eqref{Upper bound divisor} and Chebyshev's upper bound $p_m \ll m \log m$ yield in particular
\begin{equation*}
p_{N + \omega(a_0 |a_d|)} \ll \mathscr{H}(\psi_{\mathbf{a}})^{\varepsilon} N^{1+\varepsilon}.
\end{equation*}
Recalling the upper bound \eqref{Upper bound lambda}, we see that the claimed upper bound follows, which completes the proof.
\end{proof}

It may be worth pointing out that the same proof shows that the upper bound in Proposition~\ref{Proposition lambda} also holds if $d=2$.

Given $d \geq 3$ and $\psi_{\mathbf{a}} \in \mathscr{P}_d$ satisfying $0 \notin \mathrm{Prep}_{\mathbb{Q}}(\psi_{\mathbf{a}})$, Proposition~\ref{Proposition lambda} implies in particular that whenever $N > \mathscr{H}(\psi_{\mathbf{a}})^{1+\kappa}$ for some $\kappa>0$, we have
\begin{equation*}
\log N \ll \lambda_N(\psi_{\mathbf{a}}) \ll \log N,
\end{equation*}
where the implied constants depend at most on $d$ and $\kappa$.

We now proceed to study the much harder situation where $\log N$ is small compared to $\log \mathscr{H}(\psi_{\mathbf{a}})$. For this purpose, for any $t \geq 0$ we define the integer
\begin{equation}
\label{Definition Nt}
N_t(\psi_{\mathbf{a}}) = 1 + \left\lceil \mathscr{H}(\psi_{\mathbf{a}})^t \right\rceil.
\end{equation}
The following result is the culmination of our statistical investigation of the size of dynamical successive minima of polynomials.

\begin{proposition}
\label{Proposition lambda statistical}
Let $d \geq 3$ and $\varepsilon > 0$. Let also $t \geq 0$. We have
\begin{equation*}
\lim_{X \to \infty} \frac1{\# \mathscr{P}_d(X)} \cdot
\# \left\{ \psi_{\mathbf{a}} \in \mathscr{P}_d(X) :
\left( \frac{1+2t}{d+3} - \varepsilon \right) \log \mathscr{H}(\psi_{\mathbf{a}}) < \lambda_{N_t(\psi_{\mathbf{a}})}(\psi_{\mathbf{a}}) \right\} = 1.
\end{equation*}
\end{proposition}

\begin{proof}
Our goal is to prove that for any $\varepsilon \in (0, 1/2d)$, we have
\begin{equation}
\label{Goal function lower}
\! \! \! \! \lim_{X \to \infty} \frac1{\# \mathscr{P}_d(X)} \cdot
\# \left\{ \psi_{\mathbf{a}} \in \mathscr{P}_d(X) :
\lambda_{N_t(\psi_{\mathbf{a}})}(\psi_{\mathbf{a}}) \leq \left( \frac{1+2t}{d+3} - \varepsilon \right) \log \mathscr{H}(\psi_{\mathbf{a}}) \right\} = 0.
\end{equation}
First, we note that Lemma~\ref{Lemma restricting} shows that for any $\delta \in (0,1)$, we have
\begin{align}
\begin{split}
\label{Estimate restricting ell}
\# \left\{ \psi_{\mathbf{a}} \in \mathscr{P}_d(X) :
\lambda_{N_t(\psi_{\mathbf{a}})}(\psi_{\mathbf{a}}) \leq \left( \frac{1+2t}{d+3} - \varepsilon \right) \log \mathscr{H}(\psi_{\mathbf{a}}) \right\} = & \ \# \mathscr{S}_{d,\varepsilon,t}^{(\delta)}(X) \\
& + O \left( X^{d-\delta + \varepsilon} \right),
\end{split}
\end{align}
where we have set
\begin{equation*}
\mathscr{S}_{d,\varepsilon,t}^{(\delta)}(X) = \left\{ \psi_{\mathbf{a}} \in \mathscr{P}_d^{(\delta)}(X) :
\lambda_{N_t(\psi_{\mathbf{a}})}(\psi_{\mathbf{a}}) \leq \left( \frac{1+2t}{d+3} - \varepsilon \right) \log \mathscr{H}(\psi_{\mathbf{a}}) \right\}.
\end{equation*}
Next, we observe that by Definition~\ref{Definition successive}, for any $\psi_{\mathbf{a}} \in \mathscr{S}_{d,\varepsilon,t}^{(\delta)}(X)$ we clearly have
\begin{equation*}
\# \left\{ z \in \mathbb{Q}^{\times} : \hat{h}_{\psi_{\mathbf{a}}}(z) \leq \left( \frac{1+2t}{d+3} - \varepsilon \right) \log \mathscr{H}(\psi_{\mathbf{a}}) \right\} \geq N_t(\psi_{\mathbf{a}}) - 1.
\end{equation*}
Recalling the definition \eqref{Definition Nt} of the quantity $N_t(\psi_{\mathbf{a}})$, we thus see that
\begin{equation*}
\# \mathscr{S}_{d,\varepsilon,t}^{(\delta)}(X) \ll \sum_{\psi_{\mathbf{a}} \in \mathscr{P}_d^{(\delta)}(X)} \frac1{\mathscr{H}(\psi_{\mathbf{a}})^t} \cdot \# \left\{ z \in \mathbb{Q}^{\times} : \hat{h}_{\psi_{\mathbf{a}}}(z) \leq \left( \frac{1+2t}{d+3} - \varepsilon \right) \log \mathscr{H}(\psi_{\mathbf{a}}) \right\}.
\end{equation*}

We now proceed to prove that
\begin{equation}
\label{Upper bound S}
\# \mathscr{S}_{d,\varepsilon,t}^{(\delta)}(X) \ll X^{d-\varepsilon(d-1)/4+\delta(d-1)/4d}.
\end{equation}
We handle separately the cases $t=0$ and $t > 0$. First, we see that in the case $t=0$ we have
\begin{equation*}
\# \mathscr{S}_{d,\varepsilon,0}^{(\delta)}(X) \ll \sum_{\psi_{\mathbf{a}} \in \mathscr{P}_d^{(\delta)}(X)} \# \left\{ z \in \mathbb{Q}^{\times} : \hat{h}_{\psi_{\mathbf{a}}}(z) \leq \log X^{1/(d+3)-\varepsilon} \right\}.
\end{equation*}
As a result, an application of Lemma~\ref{Lemma useful} with $k=2$ gives
\begin{equation*}
\# \mathscr{S}_{d,\varepsilon,0}^{(\delta)}(X) \ll X^{d-\varepsilon(d+1)/2+\delta(d-1)/4d},
\end{equation*}
and the upper bound \eqref{Upper bound S} in the case $t=0$ follows. We now deal with the case $t > 0$. An application of partial summation gives
\begin{equation*}
\# \mathscr{S}_{d,\varepsilon,t}^{(\delta)}(X) \ll \int_1^{\infty} \sum_{\substack{\psi_{\mathbf{a}} \in \mathscr{P}_d^{(\delta)}(X) \\ \mathscr{H}(\psi_{\mathbf{a}})^t < u}} \# \left\{ z \in \mathbb{Q}^{\times} : \hat{h}_{\psi_{\mathbf{a}}}(z) \leq \left( \frac{1+2t}{d+3} - \varepsilon \right) \log \mathscr{H}(\psi_{\mathbf{a}}) \right\} \frac{\mathrm{d} u}{u^2}.
\end{equation*}
We deduce that
\begin{equation*}
\# \mathscr{S}_{d,\varepsilon,t}^{(\delta)}(X) \ll \int_1^{\infty} \! \! \! \sum_{\psi_{\mathbf{a}} \in \mathscr{P}_d^{(\delta)}(X)} \! \! \# \left\{ z \in \mathbb{Q}^{\times} : \hat{h}_{\psi_{\mathbf{a}}}(z) \leq \log \left( X^{1/(d+3)-\varepsilon/2} u^{2/(d+3)-\varepsilon/2t} \right) \right\} \frac{\mathrm{d} u}{u^2}.
\end{equation*}
Therefore, an application of Lemma~\ref{Lemma useful} with $k=2$ shows that for $t > 0$, we have
\begin{equation*}
\# \mathscr{S}_{d,\varepsilon,t}^{(\delta)}(X) \ll X^{d-\varepsilon(d-1)/4+\delta(d-1)/4d} \int_1^{\infty} \frac{\mathrm{d} u}{u^{1+\varepsilon (d+3)/4t}},
\end{equation*}
and the upper bound \eqref{Upper bound S} in the case $t>0$ follows.

In addition, since $\varepsilon \in (0, 1/2d)$ we can choose $\delta = \varepsilon d/2$. As a result, recalling the estimate \eqref{Estimate restricting ell}, we see that
\begin{equation*}
\# \left\{ \psi_{\mathbf{a}} \in \mathscr{P}_d(X) :
\lambda_{N_t(\psi_{\mathbf{a}})}(\psi_{\mathbf{a}}) \leq \left( \frac{1+2t}{d+3} - \varepsilon \right) \log \mathscr{H}(\psi_{\mathbf{a}}) \right\} \ll X^{d-\varepsilon(d-1)/8}.
\end{equation*}
Appealing to the estimate \eqref{Estimate P}, we eventually get
\begin{equation*}
\frac1{\# \mathscr{P}_d(X)} \cdot
\# \left\{ \psi_{\mathbf{a}} \in \mathscr{P}_d(X) :
\lambda_{N_t(\psi_{\mathbf{a}})}(\psi_{\mathbf{a}}) \leq \left( \frac{1+2t}{d+3} - \varepsilon \right) \log \mathscr{H}(\psi_{\mathbf{a}}) \right\} \ll \frac1{X^{\varepsilon(d-1)/8}}.
\end{equation*}
The equality \eqref{Goal function lower} thus follows, which completes the proof.
\end{proof}

We now proceed to show that Theorem~\ref{Theorem N fixed} is a direct consequence of Theorems~\ref{Theorem UBC} and \ref{Theorem Lang} and Propositions~\ref{Proposition lambda} and \ref{Proposition lambda statistical}.

\begin{proof}[Proof of Theorem~\ref{Theorem N fixed}]
To start with, we note that the upper bound in Proposition~\ref{Proposition lambda} shows that
\begin{equation*}
\# \left\{ \psi_{\mathbf{a}} \in \mathscr{P}_d(X) : \lambda_2(\psi_{\mathbf{a}}) > \left( \frac1{d-1} + \varepsilon \right) \log \mathscr{H}(\psi_{\mathbf{a}}) \right\} \leq \sum_{\substack{\psi_{\mathbf{a}} \in \mathscr{P}_d(X) \\ 0 \in \mathrm{Prep}_{\mathbb{Q}}(\psi_{\mathbf{a}})}} 1 + O(1).
\end{equation*}
It follows in particular that
\begin{equation*}
\# \left\{ \psi_{\mathbf{a}} \in \mathscr{P}_d(X) : \lambda_2(\psi_{\mathbf{a}}) > \left( \frac1{d-1} + \varepsilon \right) \log \mathscr{H}(\psi_{\mathbf{a}}) \right\} \ll \sum_{\psi_{\mathbf{a}} \in \mathscr{P}_d(X)} \# \mathrm{Prep}_{\mathbb{Q}}(\psi_{\mathbf{a}}).
\end{equation*}
Appealing to Theorem~\ref{Theorem UBC} we thus obtain
\begin{equation*}
\lim_{X \to \infty} \frac1{\# \mathscr{P}_d(X)} \cdot
\# \left\{ \psi_{\mathbf{a}} \in \mathscr{P}_d(X) : \lambda_2(\psi_{\mathbf{a}}) > \left( \frac1{d-1} + \varepsilon \right) \log \mathscr{H}(\psi_{\mathbf{a}}) \right\} = 0.
\end{equation*}
Therefore, an application of Theorem~\ref{Theorem Lang} yields
\begin{equation}
\label{Equality lambda2}
\lim_{X \to \infty} \frac1{\# \mathscr{P}_d(X)} \cdot
\# \left\{ \psi_{\mathbf{a}} \in \mathscr{P}_d(X) :
\frac{\lambda_2(\psi_{\mathbf{a}})}{\lambda_1(\psi_{\mathbf{a}})} \leq d + \varepsilon \right\} = 1.
\end{equation}
In addition, choosing $t=0$ in Proposition~\ref{Proposition lambda statistical} and appealing again to Theorem~\ref{Theorem Lang}, we derive
\begin{equation*}
\lim_{X \to \infty} \frac1{\# \mathscr{P}_d(X)} \cdot
\# \left\{ \psi_{\mathbf{a}} \in \mathscr{P}_d(X) :
\frac{d(d-1)}{d+3} - \varepsilon < \frac{\lambda_2(\psi_{\mathbf{a}})}{\lambda_1(\psi_{\mathbf{a}})} \right\} = 1.
\end{equation*}
As a result, noting that
\begin{equation*}
\frac{d(d-1)}{d+3} = d \left( 1 - \frac{4}{d+3} \right),
\end{equation*}
we see that we have obtained the first part of Theorem~\ref{Theorem N fixed}.

Next, we remark that proceeding exactly as in the proof of the equality \eqref{Equality lambda2} we get
\begin{equation*}
\lim_{X \to \infty} \frac1{\# \mathscr{P}_d(X)} \cdot
\# \left\{ \psi_{\mathbf{a}} \in \mathscr{P}_d(X) :
\frac{\lambda_N(\psi_{\mathbf{a}})}{\lambda_1(\psi_{\mathbf{a}})} \leq d + \varepsilon \right\} = 1.
\end{equation*}
Appealing to the first part of Theorem~\ref{Theorem N fixed} and noting that
\begin{equation*}
\left( 1 - \frac{4}{d+3} \right)^{-1} = 1 + \frac{4}{d-1},
\end{equation*}
we thus derive
\begin{equation*}
\lim_{X \to \infty} \frac1{\# \mathscr{P}_d(X)} \cdot
\# \left\{ \psi_{\mathbf{a}} \in \mathscr{P}_d(X) :
\frac{\lambda_N(\psi_{\mathbf{a}})}{\lambda_2(\psi_{\mathbf{a}})} \leq 1 + \frac{4}{d-1} + \varepsilon \right\} = 1.
\end{equation*}
Since we clearly have $\lambda_N(\psi_{\mathbf{a}}) \geq \lambda_2(\psi_{\mathbf{a}})$ for any $\psi_{\mathbf{a}} \in \mathscr{P}_d$, we see that this completes the proof of Theorem~\ref{Theorem N fixed}.
\end{proof}

\bibliographystyle{amsplain}
\bibliography{biblio}

\providecommand{\bysame}{\leavevmode\hbox to3em{\hrulefill}\thinspace}
\providecommand{\MR}{\relax\ifhmode\unskip\space\fi MR }
\providecommand{\MRhref}[2]{%
  \href{http://www.ams.org/mathscinet-getitem?mr=#1}{#2}
}
\providecommand{\href}[2]{#2}
\begin{thebibliography}{10}

\bibitem{MR3141413}
M.~Baker and L.~De~Marco, \emph{Special curves and postcritically finite
  polynomials}, Forum Math. Pi \textbf{1} (2013), e3, 35.

\bibitem{MR4007163}
R.~Benedetto, P.~Ingram, R.~Jones, M.~Manes, J.~H. Silverman, and T.~J. Tucker,
  \emph{Current trends and open problems in arithmetic dynamics}, Bull. Amer.
  Math. Soc. (N.S.) \textbf{56} (2019), no.~4, 611--685.

\bibitem{MR2339471}
R.~L. Benedetto, \emph{Preperiodic points of polynomials over global fields},
  J. Reine Angew. Math. \textbf{608} (2007), 123--153.

\bibitem{MR147461}
N.~G. de~Bruijn, \emph{On the number of integers {$\leq x$} whose prime factors
  divide {$n$}}, Illinois J. Math. \textbf{6} (1962), 137--141.

\bibitem{MR1995861}
N.~Fakhruddin, \emph{Questions on self maps of algebraic varieties}, J.
  Ramanujan Math. Soc. \textbf{18} (2003), no.~2, 109--122.

\bibitem{MR1480542}
E.~V. Flynn, B.~Poonen, and E.~F. Schaefer, \emph{Cycles of quadratic
  polynomials and rational points on a genus-{$2$} curve}, Duke Math. J.
  \textbf{90} (1997), no.~3, 435--463.

\bibitem{MR948108}
M.~Hindry and J.~H. Silverman, \emph{The canonical height and integral points
  on elliptic curves}, Invent. Math. \textbf{93} (1988), no.~2, 419--450.

\bibitem{MR2504779}
P.~Ingram, \emph{Lower bounds on the canonical height associated to the
  morphism {$\phi(z)=z^d+c$}}, Monatsh. Math. \textbf{157} (2009), no.~1,
  69--89.

\bibitem{MR3988672}
\bysame, \emph{Canonical heights and preperiodic points for certain weighted
  homogeneous families of polynomials}, Int. Math. Res. Not. IMRN (2019),
  no.~15, 4859--4879.

\bibitem{MR518817}
S.~Lang, \emph{Elliptic curves: {D}iophantine analysis}, Grundlehren der
  Mathematischen Wissenschaften [Fundamental Principles of Mathematical
  Sciences], vol. 231, Springer-Verlag, Berlin-New York, 1978.

\bibitem{MR4029699}
P.~Le~Boudec, \emph{A statistical view on the conjecture of {L}ang about the
  canonical height on elliptic curves}, Trans. Amer. Math. Soc. \textbf{372}
  (2019), no.~12, 8347--8361.

\bibitem{MR3953120}
N.~Looper, \emph{A lower bound on the canonical height for polynomials}, Math.
  Ann. \textbf{373} (2019), no.~3-4, 1057--1074.

\bibitem{MR4270662}
\bysame, \emph{Dynamical uniform boundedness and the {$abc$}-conjecture},
  Invent. Math. \textbf{225} (2021), no.~1, 1--44.

\bibitem{Looper2}
\bysame, \emph{The uniform boundedness and dynamical {L}ang conjectures for
  polynomials}, arXiv:2105.05240v2 (2021).

\bibitem{MR488287}
B.~Mazur, \emph{Modular curves and the {E}isenstein ideal}, Inst. Hautes
  \'{E}tudes Sci. Publ. Math. (1977), no.~47, 33--186 (1978), With an appendix
  by Mazur and M. Rapoport.

\bibitem{MR3126567}
A.~Medvedev and T.~Scanlon, \emph{Invariant varieties for polynomial dynamical
  systems}, Ann. of Math. (2) \textbf{179} (2014), no.~1, 81--177.

\bibitem{MR1199627}
P.~Morton, \emph{Arithmetic properties of periodic points of quadratic maps},
  Acta Arith. \textbf{62} (1992), no.~4, 343--372.

\bibitem{MR1264933}
P.~Morton and J.~H. Silverman, \emph{Rational periodic points of rational
  functions}, Internat. Math. Res. Notices (1994), no.~2, 97--110.

\bibitem{Panraksa}
C.~Panraksa, \emph{Rational periodic points of {$x^d+c$} and {F}ermat-{C}atalan
  equations}, arXiv:2105.03715v6 (2021).

\bibitem{MR1617987}
B.~Poonen, \emph{The classification of rational preperiodic points of quadratic
  polynomials over {${\bf Q}$}: a refined conjecture}, Math. Z. \textbf{228}
  (1998), no.~1, 11--29.

\bibitem{MR630588}
J.~H. Silverman, \emph{Lower bound for the canonical height on elliptic
  curves}, Duke Math. J. \textbf{48} (1981), no.~3, 633--648.

\bibitem{MR2316407}
\bysame, \emph{The arithmetic of dynamical systems}, Graduate Texts in
  Mathematics, vol. 241, Springer, New York, 2007.

\bibitem{MR2884382}
\bysame, \emph{Moduli spaces and arithmetic dynamics}, CRM Monograph Series,
  vol.~30, American Mathematical Society, Providence, RI, 2012.

\end{thebibliography}

\end{document}